\documentclass[a4paper,12pt,reqno]{amsart}

\usepackage{amscd,amssymb,amsmath,amsfonts,amsthm, mathrsfs,xspace,mathabx,tikz,tikz-cd,stmaryrd,url,mathtools, mathrsfs,lineno}
\usepackage{graphicx}
\usepackage{verbatim,enumerate,paralist}
\usetikzlibrary{matrix,arrows,decorations.pathmorphing}
\modulolinenumbers[5]
\usepackage{textcomp}
\usepackage[pagebackref,colorlinks]{hyperref}
\hypersetup{%
  urlcolor=blue,linkcolor=blue,citecolor=blue
}
\usepackage{pdfpages}

\DeclareMathAlphabet{\mathbbe}{U}{bbold}{m}{n}

\renewcommand{\epsilon}{\varepsilon}
\renewcommand{\phi}{\varphi}

\usepackage[T1]{fontenc}
\usepackage[all,2cell,poly,knot,tips]{xy}
\UseAllTwocells

\newcounter{Definitioncount}
\newtheorem{theorem}{Theorem}[section] 
\newtheorem{lemma}[theorem]{Lemma}
\newtheorem{definition}[theorem]{Definition}
\newtheorem{proposition}[theorem]{Proposition}
\newtheorem{corollary}[theorem]{Corollary}

\theoremstyle{definition}
\newtheorem{remark}[theorem]{Remark}
\newtheorem{example}[theorem]{Example}

\newtheoremstyle{fact}{\bigskipamount}{\medskipamount}{\upshape}{}{\itshape}{. }{ }{Fact}
\theoremstyle{fact}

\newtheoremstyle{genquest}{\bigskipamount}{\medskipamount}{\upshape}{}{\itshape}{. }{ }{General Question}
\theoremstyle{genquest}

\newtheoremstyle{step}{2\bigskipamount}{\medskipamount}{\upshape}{}{\itshape}{. }{ }{\underline{Step~\thestep}}
\theoremstyle{step}

\renewcommand{\thestep}{\arabic{step}}

\newcommand{\lra}{\longrightarrow}

\newcommand{\ra}{\rightarrow}

\newcommand{\Lra}{\Longrightarrow}

\newcommand{\ldual}[1]{\mathord{{\let\nolimits\relax\sideset{^\wedge}{}{#1}}}}
\newcommand{\laction}[2]{\mathord{{\let\nolimits\relax\sideset{^{#1}}{}{#2}}}}
\newcommand{\conj}[2]{\mathord{{\let\nolimits\relax\sideset{^{#1}}{}{#2}}}}

\newcommand{\xla}{\xleftarrow}
\newcommand{\xra}{\xrightarrow}

\def\CA{{\mathscr A}}
\def\CB{{\mathscr B}}
\def\CC{{\mathscr C}}
\def\CD{{\mathscr D}}
\def\CE{{\mathscr E}}
\def\CF{{\mathscr F}}
\def\CG{{\mathscr G}}
\def\CH{{\mathscr H}}
\def\CI{{\mathscr I}}

\def\CK{{\mathscr K}}

\def\CM{{\mathscr M}}

\def\CP{{\mathscr P}}

\def\CR{{\mathscr R}}
\def\CS{{\mathscr S}}

\def\CU{{\mathscr U}}
\def\CV{{\mathscr V}}
\def\CW{{\mathscr W}}
\def\CX{{\mathscr X}}

\newcommand{\ox}{\otimes}

\makeatletter
\newcommand{\pto}{}
\newcommand{\pgets}{}

\DeclareRobustCommand{\pto}{\mathrel{\mathpalette\p@to@gets\to}}
\DeclareRobustCommand{\pgets}{\mathrel{\mathpalette\p@to@gets\gets}}

\newcommand{\p@to@gets}[2]{%
  \ooalign{\hidewidth$\m@th#1\mapstochar\mkern5mu$\hidewidth\cr$\m@th#1\to$\cr}%
}
\makeatother

\newcommand{\xmrightarrow}[1]{\overset{#1}{\pgets}}

\begin{document}

\title{Comonadic base change for enriched categories}

\author{Branko Nikoli\'c}
\address{Department of Mathematics, Macquarie University, NSW 2109, Australia}
\email{branik.mg@gmail.com}

\author{Ross Street}
\address{Department of Mathematics, Macquarie University, NSW 2109, Australia}
\email{ross.street@mq.edu.au}

\thanks{The first author gratefully acknowledges the support of an International Macquarie University Research Scholarship while the second gratefully acknowledges the support of Australian Research Council Discovery Grants DP160101519 and DP190102432.}

\date{\today}
\maketitle

\begin{abstract}
For our concepts of change of base and comonadicity, we work in the general context of the tricategory $\mathrm{Caten}$ whose objects are bicategories $\mathscr{V}$ and whose morphisms are categories enriched on two sides. For example, for any monoidal comonad $G$ on a cocomplete closed monoidal category $\CC$, the forgetful functor $U : \CC^G\to \CC$ is comonadic when regarded as a morphism in $\mathrm{Caten}$ between one-object bicategories. Other examples are provided including that obtained from any comonoidal $\CC$-enriched category.  

We show that the forgetful pseudofunctor $\mathscr{U}:\mathscr{V}^\mathscr{G}\rightarrow \mathscr{V}$ from the bicategory of Eilenberg-Moore coalgebras for a comonad $\mathscr{G}$ 
on $\CV$ in $\mathrm{Caten}$ induces a change of base pseudofunctor $\widetilde{\mathscr{U}}:\mathscr{V}^\mathscr{G}\text{-}\mathrm{Mod}\rightarrow \mathscr{V}\text{-}\mathrm{Mod}$ which is comonadic
in a bigger version of $\mathrm{Caten}$. We should emphasise that the right adjoints to $\CU$ and $\widetilde{\CU}$ generally do not have right adjoint lax functors, confirming our need to work with two-sided enrichments.   

We define Hopfness for such a comonad $\mathscr{G}$ and prove that having
that property implies $\mathscr{U}$ creates left (Kan) extensions in the bicategory $\mathscr{V}^\mathscr{G}$.
We provide conditions under which Hopfness carries over from $\mathscr{G}$ to the comonad $\widetilde{\mathscr{G}}=\widetilde{\mathscr{U}}\circ \widetilde{\mathscr{R}}$ generated by the adjunction $\widetilde{\mathscr{U}}\dashv \widetilde{\mathscr{R}}$. This has implications for characterizing the absolute colimit completion of $\mathscr{V}^\mathscr{G}$-categories. A motivating example was the monoidal category of differential graded abelian groups obtained as the category of coalgebras for a Hopf monoid in the category of abelian groups. Examples include some involving base
bicategories $\CV = \mathrm{Spn}(\CE)$ of spans in an ordinary category $\CE$ with pullbacks.
\end{abstract}
\noindent {\small{\emph{2010 Mathematics Subject Classification:} 16E45; 16D90; 18A40}
\\
{\small{\emph{Key words and phrases:} Hopf comonad; extension creation; 2-sided enrichment; differential graded category; bicategory of spans; monoidal comonad.}}
\tableofcontents

\section*{Introduction}

Monoidal comonads $G$ on monoidal categories $\CC$ arise in many parts of mathematics.
In the case where the monoidal category $\CC$ is braided, such a $G$ can be obtained from a 
bimonoid $B$ in $\CC$: the endofunctor of the comonad is $B\ox-$, the comultiplication and unit come from the comonoid structure on $B$, while the monoid structure on $B$ makes the comonad $B\ox-$ monoidal.  

For a monoidal comonad $G$ on a monoidal category $\CC$, the category $\CC^{G}$ of
Eilenberg-Moore $G$-coalgebras is monoidal and the underlying functor $U : \CC^{G}\to \CC$
is strong monoidal; see Moerdijk \cite{Moerd02} and McCrudden \cite{McC02}. 
The present paper studies such $U$ for changing base in enriched category theory:
each category $\CA$ with homs enriched in the base $\CC^G$ yields a category $U_*\CA$ with homs enriched in the base $\CC$. 

The setting for the present paper is hom enriched category theory, as developed for example by Betti et al.\cite{22}, where the base for enrichment is a bicategory $\CV$.
In order to treat both $U$ and the base change operation in the same context, we deal with change of base bicategory of the kind developed in \cite{Kelly2002} using two-sided enrichment. 
In Section~\ref{sec:2-sided} we recall the tricategory $\mathrm{Caten}$ (bold face was originally used for this) whose objects are bicategories, whose morphisms $\CA : \CW\ra \CV$ are two-sided enriched categories, and whose 2-cells $T : \CA \Rightarrow \CB$ are enriched functors.
The homs of the tricategory are actually 2-categories, composition is 2-functorial (see (3.15) of \cite{Kelly2002}), and associativities and identities are up to isomorphism.  
It is the morphisms of $\mathrm{Caten}$, more general than lax functors, that we use to change base.

A simple base change is obtained by whiskering in $\mathrm{Caten}$ with such a morphism;
this defines a 2-functor
\begin{equation}\label{simplebasechange}
\mathrm{Caten}(\CX, \CA) = \CA\circ - : \mathrm{Caten}(\CX, \CW) \lra \mathrm{Caten}(\CX, \CV) \ .
\end{equation}
It is routine then that, if $\CA$ has a right adjoint in $\mathrm{Caten}$, then this 2-functor has a 2-adjoint. 
We also have the extension of the 2-categories $\mathrm{Caten}(\CW, \CV)$ to the bicategories
$\mathrm{Moden}(\CW, \CV)$ where the morphisms are modules instead of functors.
Recall that $\mathrm{Caten}(1, \CV)$ is the usual 2-category $\CV\text{-}\mathrm{Cat}$ of $\CV$-categories,
$\CV$-functors and $\CV$-natural transformations, while $\mathrm{Moden}(1, \CV)$ is $\CV\text{-}\mathrm{Mod}$ where the morphisms are $\CV$-modules (also called profunctors and distributors).
The base change of more interest here is the normal lax functor  
 \begin{equation}\label{interestingbasechange}
\CA\circ - : \mathrm{Moden}(\CX, \CW) \lra \mathrm{Moden}(\CX, \CV) 
\end{equation}    
of the kind occurring in Proposition 7.5 of \cite{Kelly2002}. 

Our paper has four goals. The first is to show that the lax functor \eqref{interestingbasechange} has a right
adjoint in a bigger version $\mathrm{CATEN}$ of $\mathrm{Caten}$ if $\CA$ has a right adjoint in
$\mathrm{Caten}$. This is a point that we consider to be a useful addendum to \cite{Kelly2002}.   

In Section~\ref{sec:comonads}, we study pseudocomonads $\CG$ 
in the tricategory $\mathrm{Caten}$ (recalled in Section~\ref{sec:2-sided}) of two-sided enriched categories and 
identify their Eilenberg-Moore construction. For a braided monoidal base category $\CC$, comonoidal $\CC$-enriched categories provide examples in the case where the base bicategory $\CV$ has only one object with endohom category $\CC$.  

To obtain examples at the bicategory level, in Section~\ref{sec:Spans} we take any 
pullback-preserving comonad $G$ on a category $\CE$ and construct a comonad $\CG$ on the bicategory 
$\mathrm{Spn}(\CE)$ of spans in $\CE$ whose Eilenberg-Moore construction in $\mathrm{Caten}$ is
the bicategory $\mathrm{Spn}(\CE)$ of spans in $\CE^G$. 
Recall from \cite{88} that $\mathrm{Spn}(\CE)$-enriched categories relate to indexed
(or parametrized) categories.   

Section~\ref{DiffComonads} introduces {\em differential systems} which lead, in a simple way, to
comonads in $\mathrm{Caten}$. The Eilenberg-Moore construction for the comonad can be
identified in terms of the differential system. This process abstracts that for obtaining chain complexes
from graded objects. 

Our second goal is Theorem~\ref{thm:comonadicModU} which provides mild conditions
in order for a comonadic morphism $\CU$ in $\mathrm{Caten}$ with comonad $\CG$ to induce a comonadic base
change morphism $\widetilde{\CU}$ in $\mathrm{CATEN}$ with comonad $\widetilde{\CG}$. 
In the comonoidal $\CC$-enriched category example, the Eilenberg-Moore construction generally leads to a bicategory which has more than one object.    

Section~\ref{sec:Fusion} makes explicit when a comonad in $\mathrm{Caten}$ is {\em Hopf}
and interprets it in some examples including the comonoidal $\CC$-enriched category case obtained from
a Hopf $\CC$-algebroid in the sense Definition 21 of \cite{60}. 
This leads to our third goal which is to show that Hopf comonadic morphisms in $\mathrm{Caten}$
create (preserve and reflect) left Kan extensions. Kan extensions and liftings envelope various categorical notions including adjunctions (hence duals) and (weighted) (co)limits. 
Theorem \ref{thm:createKan} generalises the works of \cite{Bruguieres2011, Chikhladze2010} 
which show that the forgetful functor from the category of algebras for a Hopf opmonoidal monad is strong closed (and a dual).
Note in passing that Weber \cite{Weber2016} describes conditions under which the forgetful functor from the bicategory of pseudo-algebras (for a 2-monad on a bicategory) reflects left Kan extensions. 

Our fourth goal is reached in Section~\ref{Hbc}: the $\widetilde{\CG}$ of Theorem~\ref{thm:comonadicModU}
is Hopf if $\CG$ is Hopf and locally cocontinuous.
We then apply the theorem to monoidal comonads whose coalgebras are graded, or differential graded, abelian groups; earlier versions appeared as \cite{Nikolic2018, NikolicPhD}.

\section{Enrichment on two sides}\label{sec:2-sided}

We begin by quickly reviewing the tricategory Caten, originally introduced and justified in \cite{Kelly2002} (where the tricategory was denoted by $\mathbf{Caten}$ while $\mathit{Caten}$ denoted the bicategory obtained on ignoring the 3-cells). The general definition of tricategory as in
\cite{GPS} is rather involved; however, Caten is mildly more complicated than a bicategory: 1-cell composition is associative up to coherent isomorphisms rather than equivalences. 

To shorten notation a little, we sometimes denote the hom category $\mathscr{V}(V,V')$ of a bicategory $\mathscr{V}$ by $\mathscr{V}_V^{V'}$. Horizontal composition in $\mathscr{V}$ will be denoted by tensor product $\otimes$. 

Objects of Caten are (small) bicategories for which we use symbols such as $\mathscr{V}$, $\mathscr{W}$, and so on. An arrows $\mathscr{A}:\mathscr{W}\rightarrow \mathscr{V}$, called a {\em 2-sided enriched category}, consists of
\begin{itemize}
\item[(i)] a set $\mathrm{Ob}\mathscr{A}$ of objects together with a span
\begin{equation*}
\begin{tikzpicture}[scale=1, every node/.style={scale=1},baseline=(current  bounding  box.center)]
\node (Pb) at (0,1) {$\mathrm{Ob}\mathscr{A}$};
\node (S) at (-1.5,0) {$\mathrm{Ob}\mathscr{W}$};
\node (P) at (1.5,0) {$\mathrm{Ob}\mathscr{V}$};
\path[->,font=\normalsize,>=angle 90]
(Pb) edge node[above left] {$(-)_{-}$} (S);
\path[->,font=\normalsize,>=angle 90] 
(Pb) edge node[above right] {$(-)_{+}$} (P);
\end{tikzpicture}
\end{equation*}
assigning to each object $A$ an object $A_-$ in $\mathscr{W}$ and $A_+$ in $\mathscr{V}$, 
\item[(ii)] homs $\mathscr{A}(A,A')$, also denoted $\mathscr{A}_A^{A'}$, defined to be functors
\begin{equation*}
\mathscr{A}_A^{A'}:\mathscr{W}_{A_-}^{A'_-}\rightarrow \mathscr{V}_{A_+}^{A'_+} \ ,
\end{equation*}
\item[(iii)] unit and composition natural transformations
\begin{equation*}
\begin{tikzpicture}[xscale=1.2, yscale=0.8, every node/.style={scale=1},baseline=(current  bounding  box.center)]
\node (A1) at (-1,1) {$1$};
\node (A2) at (1,1) {$\mathscr{W}_{A_-}^{A_-}$};
\node (B2) at (1,-1) {$\mathscr{V}_{A_+}^{A_+}$};
\node (mu) at (0.2,0) {$\scriptstyle\eta_{A}\Rightarrow$};
\path[->,font=\scriptsize,>=angle 90]
(A1) edge node[above] {$1_{A_-}$} (A2);
\path[->,font=\scriptsize,>=angle 90, bend right]
(A1) edge node[below left] {$1_{A_+}$} (B2);
\path[,->,font=\scriptsize,>=angle 90]
(A2) edge node[right] {$\mathscr{A}_{A}^{A}$} (B2);
\end{tikzpicture}
\begin{tikzpicture}[xscale=1.6, yscale=0.8, every node/.style={scale=1},baseline=(current  bounding  box.center)]
\node (A1) at (-1,1) {$\mathscr{W}_{A'_-}^{A^{\prime\prime}_-}\times \mathscr{W}_{A_-}^{A'_-}$};
\node (A2) at (1,1) {$\mathscr{W}_{A_-}^{A^{\prime\prime}_-}$};
\node (B1) at (-1,-1) {$\mathscr{V}_{A'_+}^{A^{\prime\prime}_+}\times \mathscr{V}_{A_+}^{A'_+}$};
\node (B2) at (1,-1) {$\mathscr{V}_{A_+}^{A^{\prime\prime}_+}$};
\node (mu) at (0.2,0) {$\scriptstyle\mu_{AA''}^{A'}\Rightarrow$};
\path[->,font=\scriptsize,>=angle 90]
(A1) edge node[above] {$\otimes$} (A2)
(B1) edge node[below] {$\otimes$} (B2);
\path[transform canvas={xshift=9mm},->,font=\scriptsize,>=angle 90]
(A1) edge node[left] {$\mathscr{A}_{A'}^{A''}\times \mathscr{A}_A^{A'}$} (B1);
\path[transform canvas={xshift=-3mm},->,font=\scriptsize,>=angle 90]
(A2) edge node[right] {$\mathscr{A}_{A}^{A''}$} (B2);
\end{tikzpicture}
\end{equation*}
satisfying unit and associativity laws.
\end{itemize}
Composition of 2-sided enriched categories is given by composition of the spans (which uses pullback \cite{Ben1967}), composition of the functors defining homs, and pasting of the unit and multiplication natural transformations.
\begin{example}
When $\mathscr{W}=1$, $\mathscr{A}$ is precisely a category enriched in the bicategory $\mathscr{V}$. 
\end{example}
\begin{example}\label{one_object_all}
An enriched category $\CA : \CW\to \CV$ with precisely one object might be called a {\em monad} from $\CW$ to $\CV$. If furthermore the 
bicategories $\CW$ and $\CV$ also each have precisely one object then such an $\CA$ amounts to a monoidal functor $T : \CD\to \CC$ between monoidal categories (in the sense of \cite{EilKel1966}) where $\CC$ and $\CD$ are the 
endohom monoidal categories of the single objects of $\CV$ and $\CW$. When $\CD = \mathbf{1}$ then $T$ is
precisely a monoid in $\CC$.
\end{example}
\begin{example}\label{action}
Let $\CK$ be a bicategory with two full subbicategories $\CM$ and $\CP$.
Suppose $\Lambda$ is a set of morphisms $a\in \CK(M,P)$ with $M\in \CM$ and $P\in \CP$ such that each functor
$\CK(a,P') : \CK(P,P') \to \CK(M,P')$ with $P'\in \CP$ has a right adjoint $\mathrm{ran}(a,-)$.  
We can define an enriched category $\CA : \CM\to \CP$ having
$\mathrm{ob}\CA = \Lambda$ with $a_-=M\in \CM$ and $a_+=P\in \CP$.
The functor $\CA_a^{a'}$ is the composite $\CK(M,M')\xra{\CK(M,a')}\CK(M,P')\xra{\mathrm{ran}(a,-)}\CK(P,P')$.
The unit $\eta_a$ and composition $\mu_{a a''}^{a'}$ are obtained by using the universal property of the right adjoints 
$\mathrm{ran}(a,-)$ and $\mathrm{ran}(a',-)$. 
Example 2.3 (f) of \cite{Kelly2002} about two-sided monoidal actions is a special case
where $\CM$ and $\CP$ each have one object and $\CK$ has two objects.    
\end{example}

A 2-cell $F:\mathscr{A}\rightarrow \mathscr{B}$, called an {\em (enriched) functor}, consists of:
\begin{itemize}
\item[(iv)] a map of spans 
\begin{equation*}
F:\mathrm{Ob}\mathscr{A}\rightarrow \mathrm{Ob}\mathscr{B}
\end{equation*}
which means a function between the object sets such that
\begin{equation}\label{eq:obComp}
(FA)_-=A_-\;\text{and}\;(FA)_+=A_+
\end{equation}
\item[(v)] natural transformations
\begin{equation*}
F_A^{A'}:\mathscr{A}_A^{A'}\Rightarrow \mathscr{B}_{FA}^{FA'}
\end{equation*}
which are compatible with the unit and multiplication of $\mathscr{A}$ and $\mathscr{B}$.
\end{itemize}

A 3-cell $\psi:F\rightarrow E$, called an {\em (enriched) natural transformation} is a family with components 2-cells
\begin{equation*}
\psi_A:1_{A_+}\Rightarrow \mathscr{B}_{FA}^{EA}(1_{A_-})
\end{equation*}
in $\CV$ satisfying an enriched naturality condition (the coherence 2-cells we omit are set out in \cite{Kelly2002})
\begin{equation}\label{diag:enrNat}
\begin{tikzpicture}[xscale=3, yscale=0.8, every node/.style={scale=1},baseline=(current  bounding  box.center)]
\node (A1) at (-1,1) {$\mathscr{A}_A^{A'}(w)$};
\node (A2) at (1,1) {$\mathscr{B}_{FA'}^{EA'}(1_{A'_-})\otimes \mathscr{B}_{FA}^{FA'}(w)$};
\node (B1) at (-1,-1) {$ \mathscr{B}_{EA}^{EA'}(w)\otimes\mathscr{B}_{FA}^{EA}(1_{A_-})$};
\node (B2) at (1,-1) {$\mathscr{B}_{FA}^{EA'}(w)\,.$};
\path[->,font=\scriptsize,>=angle 90]
(A1) edge node[above] {$\psi_{A'}\otimes (F_A^{A'})_w$} (A2)
(B1) edge node[below] {$\mu$} (B2);
\path[->,font=\scriptsize,>=angle 90]
(A1) edge node[left] {$ (E_A^{A'})_w\otimes\psi_A$} (B1);
\path[->,font=\scriptsize,>=angle 90]
(A2) edge node[right] {$\mu$} (B2);
\end{tikzpicture}
\end{equation}

All axioms, compositions, whiskerings, and the fact that Caten is a tricategory are explained in detail in \cite{Kelly2002}. This construction in a larger universe, which contains our ambient category of sets, is denoted CATEN. Theorems below hold when Caten is substituted with CATEN, see the discussion on page 56 of \cite{Kelly2002}. 

As mentioned, it is important to realise that Caten is actually a fairly special kind of tricategory.
Without the 3-cells it is a bicategory, while the homs Caten$(\CW,\CV)$ are 2-categories.  
\begin{example}
When $\mathrm{Ob}\mathscr{A}=\mathrm{Ob}\mathscr{W}$ and $(-)_-$ is the identity function, $\mathscr{A}$ is precisely a lax functor from $\mathscr{W}$ to $\mathscr{V}$, and 2-cells are `icons' as so named in \cite{Lack2010a}.
\end{example}

\begin{remark}\label{adjinCaten}
Recall that Proposition 2.7 of \cite{Kelly2002} characterizes left adjoints
in $\mathrm{Caten}$ as those enriched categories $\CA : \CW\to \CV$ which are pseudofunctors
and are local left adjoints (that is, each functor $\CA^{A'}_A : \CW^{A'_-}_{A_-}\to \CV^{A'_+}_{A_+}$
has a right adjoint in $\mathrm{Cat}$).   
\end{remark}

\begin{example}\label{restrictedrepresentable}
Let $K$ be an object of a bicategory $\CK$ and let $\CM$ be a full subbicategory of $\CK$.
Suppose $\CX$ is a sub-pseudofunctor of the restricted representable pseudofunctor $\CM \hookrightarrow \CK\xra{\CK(K,-)}\mathrm{Cat}$ such that the right extension $\mathrm{ran}(x,x') : A\to B$ of $x'$ along $x$ exists 
(see Example~\ref{action} for the notation) for all $A, B\in \CM$, all $x\in \CX(A)$ and $x'\in \CX(B)$. 
If each functor $\CX(A)^{\mathrm{op}}\times \CX(A) \to \CM(A,B)$
admits an end then $\CX : \CM \to \mathrm{Cat}$ admits a right adjoint in $\mathrm{Caten}$.
Indeed, the right adjoint to the functor $\CX_A^B : \CM(A,B)\to [\CX(A),\CX(B)]$ takes the functor 
$h : \CX(A)\to \CX(B)$ to the end $\int_{x\in \CX(A)}\mathrm{ran}(x,h(x))$.  
In particular, $\CK$ could be any bicategory whose homs are complete lattices and which admits all right extensions,
and where $K$ is any object, $\CM = \CK$ and $\CX = \CK(K,-)$. Even more particularly, $\CK$ could be the bicategory
$\mathrm{Rel}(\CE)$ of relations in any Grothendieck topos $\CE$. 
\end{example}

Instead of (enriched) functors we could have chosen {\em enriched modules} $M:\mathscr{A}\xmrightarrow{}\mathscr{B}$ as 2-cells. Such a module consists of
\begin{itemize}
\item[(vi)]  
functors
\begin{equation*}
M_B^A:\mathscr{W}_{B_-}^{A_-}\rightarrow \mathscr{V}_{B_+}^{A_+}
\end{equation*}
\item[(vii)] action natural transformations
\begin{equation}\label{diag:actions}
\begin{aligned}
\begin{tikzpicture}[xscale=1.5, yscale=0.8, every node/.style={scale=1},baseline=(current  bounding  box.center)]
\node (A1) at (-1,1) {$\mathscr{W}_{A_-}^{A'_-}\times \mathscr{W}_{B_-}^{A_-}$};
\node (A2) at (1,1) {$\mathscr{W}_{B_-}^{A'_-}$};
\node (B1) at (-1,-1) {$\mathscr{V}_{A_+}^{A'_+}\times \mathscr{V}_{B_+}^{A_+}$};
\node (B2) at (1,-1) {$\mathscr{V}_{B_+}^{A'_+}$};
\node (mu) at (0.2,0) {$\scriptstyle\lambda_{BA'}^{A}\Rightarrow$};
\path[->,font=\scriptsize,>=angle 90]
(A1) edge node[above] {$\otimes$} (A2)
(B1) edge node[below] {$\otimes$} (B2);
\path[transform canvas={xshift=9mm},->,font=\scriptsize,>=angle 90]
(A1) edge node[left] {$\mathscr{A}_{A}^{A'}\times M_B^{A}$} (B1);
\path[transform canvas={xshift=-3mm},->,font=\scriptsize,>=angle 90]
(A2) edge node[right] {$M_{B}^{A'}$} (B2);
\end{tikzpicture}
\begin{tikzpicture}[xscale=1.5, yscale=0.8, every node/.style={scale=1},baseline=(current  bounding  box.center)]
\node (A1) at (-1,1) {$\mathscr{W}_{B'_-}^{A_-}\times \mathscr{W}_{B_-}^{B'_-}$};
\node (A2) at (1,1) {$\mathscr{W}_{B_-}^{A_-}$};
\node (B1) at (-1,-1) {$\mathscr{V}_{B'_+}^{A_+}\times \mathscr{V}_{B_+}^{B'_+}$};
\node (B2) at (1,-1) {$\mathscr{V}_{B_+}^{A_+}$};
\node (mu) at (0.2,0) {$\scriptstyle\rho_{BA}^{B'}\Rightarrow$};
\path[->,font=\scriptsize,>=angle 90]
(A1) edge node[above] {$\otimes$} (A2)
(B1) edge node[below] {$\otimes$} (B2);
\path[transform canvas={xshift=9mm},->,font=\scriptsize,>=angle 90]
(A1) edge node[left] {$M_{B'}^{A}\times \mathscr{B}_B^{B'}$} (B1);
\path[transform canvas={xshift=-3mm},->,font=\scriptsize,>=angle 90]
(A2) edge node[right] {$M_{B}^{A}$} (B2);
\end{tikzpicture}
\end{aligned}
\end{equation}
compatible with each other, and with the units and compositions in $\mathscr{A}$ and $\mathscr{B}$.
\end{itemize}

A {\em module morphism} $\sigma:M\Rightarrow N$ consists of natural transformations
\begin{equation*}
\sigma_B^A:M_B^A\Rightarrow N_B^A
\end{equation*}
compatible with the actions (\ref{diag:actions}).

Module morphisms compose and we obtain a category $\mathrm{Moden}(\mathscr{W},\mathscr{V})(\mathscr{A},\mathscr{B})$ of modules between $\mathscr{A}$ and $\mathscr{B}$. 

When $\mathscr{V}$ is locally cocomplete (including having whiskering preserve the local colimits), $\mathrm{Moden}(\mathscr{W},\mathscr{V})$ becomes a bicategory. Composition of modules is defined by the reflexive coequalizer
\begin{eqnarray}\label{modulecompo}
\xymatrix @R-3mm {
\sum_{B,B'}{M^A_{B'}\ox \CB^{B'}_B \ox N^B_C} \ar@<1.5ex>[rr]^{\phantom{aaa} \sum_B \rho\ox 1}  \ar@<-1.5ex>[rr]_{\phantom{aaa} \sum_{B'}1\ox \lambda} && \sum_B{M^A_B\ox N^B_C} \ar@<0.0ex>[rr]^{\mathrm{coeq}} && (N\circ_{\CB} M)_C^A \ . 
}\end{eqnarray}
In particular,  $\mathrm{Moden}(1,\mathscr{V})$ is the usual bicategory $\CV\text{-}\mathrm{Mod}$
of $\CV$-categories and $\CV$-modules between them. In general, 
$\mathrm{Moden}(\mathscr{W},\mathscr{V})$ is equivalent to the bicategory $\mathrm{Conv}(\mathscr{W},\mathscr{V})\text{-}\mathrm{Mod}$, where $\mathrm{Conv}$ denotes the internal hom in $\mathrm{CATEN}$ for the cartesian product of bicategories; see \cite{Kelly2002}. The bicategory $\mathrm{Conv}(\CX,\CV)$ is constructed by Day local convolution \cite{Day1981} as follows:
\begin{eqnarray*}
\mathrm{obConv}(\CX,\CV) = \mathrm{ob}\CX \times \mathrm{ob}\CV  \text{  and  }
\mathrm{Conv}(\CX,\CV)((X,V),(X',V') = [\CX(X,X'),\CV(V,V')]
\end{eqnarray*}
 with composition functors
 \begin{eqnarray*}
[\CX(X',X''),\CV(V',V'')] \times [\CX(X,X'),\CV(V,V')] \xrightarrow{\bar{\otimes}} [\CX(X,X''),\CV(V,V'')]
\end{eqnarray*}
defined by the coend formula
\begin{eqnarray}\label{convolutioncoend}
P\bar{\otimes}Q = \int^{h.k}{\CX(X,X'')(h\otimes k,-)\bullet (P(h)\otimes Q(k))} \ .
\end{eqnarray}

\begin{remark}\label{reflexivecoequsonly}
If in \eqref{modulecompo} the enriched category $\CB$ has only one object then we only need the existence of 
local reflexive coequalizers in $\CV_{C_+}^{A_+}$.  
\end{remark}
\begin{remark}\label{monoidcase}
As noted in Example~\ref{one_object_all}, monoidal functors can be regarded as two-sided enriched categories. 
If $T, S : \CD\to \CC$ are monoidal functors between monoidal categories, a module $M : T\to S$ amounts to a
functor $M :\CD\to \CC$ equipped with natural families $\lambda_{D,D'} : TD\ox MD'\to M(D\ox D')$ and 
$\rho_{D',D''} : MD'\ox SD''\to M(D'\ox D'')$ providing left and right compatible actions. 
In this situation, Remark~\ref{reflexivecoequsonly} applies to composition $T\xra{M} S\xra{N} R$: we only require the existence of reflexive coequalizers in $\CC$, which are preserved by tensoring by an object on either side, rather than requiring all colimits.
Let us denote this bicategory of monoidal functors and modules by $\mathrm{Mod1}(\CD,\CC)$.
When $\CD=\mathbf{1}$, we obtain the usual bicategory $\mathrm{Mod}(\CC)$ whose objects are monoids in $\CC$ 
(called ``rings'' or ``algebras'' in the additive case) and bimodules as morphisms.   
\end{remark}

In \cite{Kelly2002}~Example 7.4(b), a module $\mathscr{X}(T,S)$ was constructed from a cospan $\mathscr{A}\xrightarrow{S}\mathscr{X}\xleftarrow{T}\mathscr{B}$ of functors.
We will discuss the case $F_{*} = \mathscr{B}(1_{\mathscr{B}},F)$ in a bit more detail. 
Each functor $F:\mathscr{A}\rightarrow\mathscr{B}$ 
defines a module $F_{*}:\mathscr{A}\xmrightarrow{}\mathscr{B}$ by taking 
\begin{align}
(F_*)_B^A=&\mathscr{B}_B^{FA}
	:\mathscr{W}_{B_-}^{A_-}\rightarrow \mathscr{V}_{B_+}^{A_+} \nonumber\\
\lambda_{BA'}^{A}= &\otimes(\mathscr{A}_{A}^{A'}\times \mathscr{B}_B^{FA})
	\xRightarrow{1(F_A^{A'}\times 1)}
		\otimes(\mathscr{B}_{FA}^{FA'}\times \mathscr{B}_B^{FA})
		\xRightarrow{\mu_{B,FA'}^{FA}}
			\mathscr{B}_{B}^{FA'}\otimes \label{eq:indModLam}\\
\rho_{BA}^{B'}= &\otimes(\mathscr{B}_{B'}^{FA}\times \mathscr{B}_{B}^{B'})
		\xRightarrow{\mu_{B,FA}^{B'}}
			\mathscr{B}_{B}^{FA}\otimes	\,\label{eq:indModRho}
\end{align}
which is properly typed because of (\ref{eq:obComp}).
The unit and associativity axioms for the $\mu$ and $\eta$ of $\mathscr{B}$ give their compatibility with $\rho$. Compatibility of $\lambda$ with the $\mu$ and $\eta$ of $\mathscr{A}$ follows by applying their compatibility with the functor $F$ followed by their unit and associativity laws.

Similarly, each natural transformation $\psi:F\rightarrow E$ has an induced module morphism
\begin{align}
\psi_*:F_*&\rightarrow E_*\label{eq:natToMor1} \\
(\psi_*)_B^A:=\mathscr{B}_B^{FA}&\Rightarrow \mathscr{B}_B^{EA} \nonumber\\
((\psi_*)_B^A)_w:=\mathscr{B}_B^{FA}(w)&\xrightarrow{\psi_A\otimes 1} 
	\mathscr{B}_{FA}^{EA}(1_{A_-})\otimes\mathscr{B}_B^{FA}(w)\xrightarrow{\mu_{B,EA}^{FA} }\mathscr{B}_B^{EA}(w) \nonumber\,. 
\end{align}
To see that $\psi_*$ is compatible with $\lambda$, tensor diagram (\ref{diag:enrNat}) by $\mathscr{B}_B^{FA}(w')$, 
whisker the resulting square on the right with the $\mu$ of $\mathscr{B}$, and add obvious commutative squares. 
Compatibility with $\rho$ follows from associativity of the $\mu$ for $\mathscr{B}$.
Also, every module morphism between modules induced by functors  gives rise to a natural transformation: given
\begin{equation*}
\sigma_B^A:\mathscr{B}_B^{FA}\Rightarrow \mathscr{B}_B^{EA}
\end{equation*}
we can define a natural transformation $\sigma$ with components
\begin{equation}\label{eq:morToNat}
\sigma_A:1_{A_+}\xrightarrow{\eta_{FA}}\mathscr{B}_{FA}^{FA}(1_{A_-})\xrightarrow{(\sigma_{FA}^{A})_{1_{A_-}}} \mathscr{B}_{FA}^{EA}(1_{A_-})
\end{equation}
and the natural transformation axiom (\ref{diag:enrNat}) is proved by commutativity of
\begin{equation*}
\begin{tikzpicture}[xscale=2,yscale=1.1, every node/.style={scale=0.8},baseline=(current  bounding  box.center)]
\node (11) at (-0.5,0) {$\mathscr{A}_A^{A'}(w)$};
\node (12) at (0,-1.5) {$\mathscr{A}_A^{A'}(w)\otimes \mathscr{B}_{FA}^{FA}(1_{A_-})$};
\node (13) at (0,-4) {$\mathscr{A}_A^{A'}(w)\otimes \mathscr{B}_{FA}^{EA}(1_{A_-})$};

\node (21) at (1.5,0) {$\mathscr{B}_{FA'}^{FA'}(1_{A'_-})\otimes\mathscr{A}_A^{A'}(w)$};
\node (23) at (2.5,-4) {$\mathscr{B}_{EA}^{EA'}(w)\otimes\mathscr{B}_{FA}^{EA'}(1_{A_-}) $};

\node (31) at (4,0) {$\mathscr{B}_{FA'}^{EA'}(1_{A'_-})\otimes\mathscr{A}_A^{A'}(w)$};
\node (32) at (4,-2.5) {$\mathscr{B}_{FA'}^{EA'}(1_{A'_-})\otimes\mathscr{B}_{FA}^{FA'}(w)$};
\node (33) at (4.5,-4) {$\mathscr{B}_{FA}^{EA'}(w)$};

\node (i11) at (1,-1) {$\mathscr{B}_{FA}^{FA'}(w)$};
\node (i12) at (1.5,-3) {$\mathscr{B}_{FA}^{FA'}(w)\otimes\mathscr{B}_{FA}^{FA}(1_{A_-})$};
\node (i21) at (2.5,-1) {$\mathscr{B}_{FA'}^{FA'}(1_{A'_-})\otimes\mathscr{B}_{FA}^{FA'}(w)$};
\node (i22) at (3,-3) {$\mathscr{B}_{FA}^{FA'}(w)$};

\path[->,font=\scriptsize,>=angle 90]
(11) edge node[left] {$1\otimes\eta_{FA}$} (12)
(12) edge node[left] {$1\otimes(\sigma_{FA}^A)_{1_{A_-}}$} (13)
(13) edge node[below] {$(E_A^{A'})_w\otimes 1$} (23)
(23) edge node[below] {$\mu$} (33)
;
\path[->,font=\scriptsize,>=angle 90]
(11) edge node[above] {$\eta_{FA'}\otimes 1$} (21)
(21) edge node[above] {$(\sigma_{FA'}^{A'})_{1_{A'_-}}\otimes1$} (31)
(31) edge node[right] {$1\otimes (F_A^{A'})_w$} (32)
(32) edge node[right] {$\mu$} (33)
;
\path[->,font=\scriptsize,>=angle 90]
(11) edge node[below] {$(F_A^{A'})_w$} (i11)
(12) edge node[below left] {$(F_A^{A'})_w\otimes 1$} (i12)
(21) edge node[above right] {$1\otimes (F_A^{A'})_w$} (i21)
(i21) edge node[fill=white,rounded corners=1pt,inner sep=0.3pt] {$(\sigma_{FA'}^{A'})_{1_{A'_-}}\otimes 1$} (32)
(i22) edge node[above right] {$(\sigma_{FA}^{A'})_{w}$} (33)
;
\path[->,font=\scriptsize,>=angle 90]
(i11) edge node[above] {$\eta_{FA'}\otimes 1$} (i21)
(i21) edge node[left] {$\mu$} (i22)
(i11) edge node[left] {$1\otimes \eta_{FA}$} (i12)
(i12) edge node[below] {$\mu$} (i22)
(i11) edge node[fill=white,rounded corners=1pt,inner sep=0.3pt] {$1$} (i22)
;
\end{tikzpicture}
\end{equation*}
where the hexagon and the bottom right square are compatibility conditions between module morphism $\sigma$ and actions (\ref{eq:indModLam}) and (\ref{eq:indModRho}) respectively.

\begin{proposition}\label{prop:2-sidY}
The functor 
\begin{equation*}
(-)_*:\mathrm{Caten}(\mathscr{W},\mathscr{V})(\mathscr{A},\mathscr{B})
\rightarrow
\mathrm{Moden}(\mathscr{W},\mathscr{V})(\mathscr{A},\mathscr{B})
\end{equation*}
is full and faithful.
\end{proposition}
\begin{proof}
The process \eqref{eq:natToMor1} of turning a natural transformation into a module morphism has inverse defined by \eqref{eq:morToNat}, as shown by the commuting diagrams \eqref{diag:yComm}.
\end{proof}
\begin{equation}\label{diag:yComm}
\begin{aligned}
\begin{tikzpicture}[xscale=4, yscale=1.5, every node/.style={scale=0.8},baseline=(current  bounding  box.center)]
\node (A1) at (0,0) {$1_{A_+}$};
\node (A2) at (1,0) {$\mathscr{B}_{FA}^{FA}(1_{A_-})$};
\node (B1) at (0,-1) {$\mathscr{B}_{FA}^{EA}(1_{A_-})$};
\node (B2) at (1,-1) {$\mathscr{B}_{FA}^{EA}(1_{A_-})\otimes\mathscr{B}_{FA}^{FA}(1_{A_-})$};
\node (C) at (1,-2) {$\mathscr{B}_{FA}^{EA}(1_{A_-})$};
\path[->,font=\scriptsize,>=angle 90]
(A1) edge node[above] {$\eta_{FA}$} (A2)
(B1) edge node[above] {$1\otimes \eta_{FA}$} (B2)
(B1) edge node[below] {$1$} (C);
\path[->,font=\scriptsize,>=angle 90]
(A1) edge node[left] {$\psi_A$} (B1);
\path[->,font=\scriptsize,>=angle 90]
(A2) edge node[right] {$\psi_A\otimes 1$} (B2)
(B2) edge node[right] {$\mu$} (C);
\end{tikzpicture}
\begin{tikzpicture}[xscale=3.2, yscale=1.5, every node/.style={scale=0.8},baseline=(current  bounding  box.center)]
\node (C) at (0,0) {$\mathscr{B}_{B}^{FA}(w)$};
\node (B2) at (0,-1) {$\mathscr{B}_{FA}^{FA}(1_{A_-})\otimes\mathscr{B}_{B}^{FA}(w)$};
\node (B1) at (1,-1) {$\mathscr{B}_{B}^{FA}(w)$};
\node (A1) at (1,-2) {$\mathscr{B}_{B}^{EA}(w)$};
\node (A2) at (0,-2) {$\mathscr{B}_{FA}^{EA}(1_{A_-})\otimes\mathscr{B}_{B}^{FA}(w)$};

\path[<-,font=\scriptsize,>=angle 90]
(A1) edge node[below] {$\mu$} (A2)
(B1) edge node[below] {$\mu$} (B2)
(B1) edge node[above] {$1$} (C);
\path[<-,font=\scriptsize,>=angle 90]
(A1) edge node[right] {$(\sigma_B^A)_w$} (B1);
\path[<-,font=\scriptsize,>=angle 90]
(A2) edge node[left] {$(\sigma_{FA}^A)_{1_{A_-}}\otimes 1$} (B2)
(B2) edge node[left] {$\eta_{FA}\otimes 1$} (C);
\end{tikzpicture}
\end{aligned}
\end{equation}

We end this section by recalling the {\bf duality} for two-sided enriched categories defined in Section 2.9 of 
\cite{Kelly2002}. There is an involutory trihomomorphism (in the sense of \cite{GPS})
\begin{eqnarray*}
(-)^{\circ} : \mathrm{Caten} \lra \mathrm{Caten} 
\end{eqnarray*}
which is a pseudofunctor between bicategories when 3-cells are ignored.
The involution is defined on objects $\CV$ by $\CV\mapsto \CV^{\mathrm{op}}$.

\section{Base change under a left adjoint in $\mathrm{Caten}$}

As mentioned in the Introduction, a base change 2-functor of the form \eqref{simplebasechange} will have a right adjoint 2-functor if the base change morphism has a right adjoint in $\mathrm{Caten}$. The simple reason
no longer applies when we ask for a $\mathrm{Moden}$ version. Indeed, we do prove that there is a right
adjoint to the induced pseudofunctor although it is an adjoint enriched category, not an adjoint pseudofunctor.
  
Recall the base change result in \cite{Kelly2002} which appears there as Proposition 7.5 and states that right whiskering modules by a two-sided enriched category provides a normal lax functor.
\begin{theorem}\label{thm:adjointModU}
Let $\CU : \CW\to \CV$ be a morphism in $\mathrm{Caten}$ between locally cocomplete bicategories and suppose ${\CU}\dashv {\CR}$ is an adjunction in $\mathrm{Caten}$ with unit $\mathrm{g} : 1_{\CV}\to \CR\CU$. 
Then the lax functor 
\begin{align*}
\mathscr{N}:= \mathrm{Moden}(\mathscr{X},\mathscr{W})
 \xrightarrow{\widetilde{\mathscr{U}}:=\mathrm{Moden}(\mathscr{X},\mathscr{U})}
 \mathscr{M:=}\mathrm{Moden}(\mathscr{X},\mathscr{V}) \ ,
\end{align*}
given by right whiskering with ${\CU}$, has a right adjoint $\widetilde{\CR}$ in\footnote{Note that $\mathscr{N}$ and $\mathscr{M}$ may have a large set (proper class) of objects.} $\mathrm{CATEN}$. 
\end{theorem}
\begin{proof}
We will consider the case when $\mathscr{X}$ is the terminal bicategory; the general case then follows from $\mathrm{Moden}(\mathscr{W},\mathscr{V})\simeq \mathrm{Conv}(\mathscr{W},\mathscr{V})\text{-}\mathrm{Mod}$ and Proposition~\ref{Convcoalg}. 

So we are looking at $\mathscr{N}=\mathscr{W}\text{-Mod}$ and $\mathscr{M}=\mathscr{V}\text{-Mod}$.

First we show that $\widetilde{\mathscr{U}}$ has local right adjoints 
$\widetilde{\mathscr{R}}_\mathscr{A}^\mathscr{B} : \mathscr{M}(\mathscr{U}\circ \mathscr{A},\mathscr{U}\circ \mathscr{B}) \longrightarrow \mathscr{N}(\mathscr{A},\mathscr{B})$ given by 
\begin{align*}
(\mathscr{U}\circ \mathscr{A}\xmrightarrow{M}\mathscr{U}\circ \mathscr{B},\alpha) &\mapsto ( \mathscr{A}\xmrightarrow{\widetilde{\mathscr{R}}M} \mathscr{B},\widetilde{\mathscr{R}}\alpha)\\
(\sigma:M\Rightarrow N)&\mapsto (\widetilde{\mathscr{R}}\sigma:\widetilde{\mathscr{R}}M\Rightarrow \widetilde{\mathscr{R}}N)
\end{align*}
with the assignments defined by
\begin{align*}
(\widetilde{\mathscr{R}}M)_B^A&:=\mathscr{R}M_B^A\\
(\widetilde{\mathscr{R}}\alpha)_{B,A'}^{B',A}&:=(\mathscr{A}_A^{A'}\otimes\mathscr{R}M_{B'}^A\otimes\mathscr{B}_B^{B'}
\xrightarrow{\mathrm{g}\otimes 1\otimes \mathrm{g}} 
\mathscr{R}\mathscr{U}\mathscr{A}_A^{A'}
\otimes \mathscr{R}M_{B'}^A
\otimes \mathscr{R}\mathscr{U}\mathscr{B}_B^{B'} \\
&\xrightarrow{\mu^{(\mathscr{R})}} 
\mathscr{R}(\mathscr{U}\mathscr{A}_A^{A'}
\otimes M_{B'}^A
\otimes \mathscr{U}\mathscr{B}_B^{B'})
\xrightarrow{\mathscr{R}(\alpha)} 
\mathscr{R}M_B^{A'})\\
(\widetilde{\mathscr{R}}\sigma)_B^{A}&:=\mathscr{R}\sigma_B^{A}=\sigma_B^A\,.
\end{align*}
where $\alpha$ denotes a 2-sided action (the analogous 1-sided actions are denoted by $\lambda$ and $\rho$).
Actions $\widetilde{\mathscr{R}}\alpha$, (or separately $\widetilde{\mathscr{R}}\lambda$ and $\widetilde{\mathscr{R}}\rho$) are compatible with unit and composition in $\mathscr{A}$ and $\mathscr{B}$. For example, compatibility of $\rho$ with composition is shown by commutativity of the diagram (\ref{diag:rhoMuComp}).
\begin{equation}\label{diag:rhoMuComp}
\begin{tikzpicture}[xscale=2,yscale=2, every node/.style={scale=0.8},baseline=(current  bounding  box.center)]
\node (11) at (0,0) {$\mathscr{R}M_{B''}^A\otimes \mathscr{B}_{B'}^{B''}\otimes \mathscr{B}_B^{B'}$};
\node (12) at (0,-1) {$\mathscr{R}M\otimes \mathscr{R}\mathscr{U}\mathscr{B}\otimes \mathscr{B}$};
\node (13) at (0,-2) {$\mathscr{R}(M\otimes \mathscr{U}\mathscr{B})\otimes \mathscr{B}$};
\node (14) at (0,-3) {$\mathscr{R}M\otimes \mathscr{B}$};

\node (24) at (1.5,-3) {$\mathscr{R}M\otimes \mathscr{R}\mathscr{U}\mathscr{B}$};
\node (34) at (3,-3) {$\mathscr{R}(M\otimes \mathscr{U}\mathscr{B})$};
\node (44) at (4.5,-3) {$\mathscr{R}\mathscr{B}$};

\node (41) at (4.5,0) {$\mathscr{R}M\otimes \mathscr{B}$};
\node (42) at (4.5,-1) {$\mathscr{R}M\otimes \mathscr{R}\mathscr{U}\mathscr{B}$};
\node (43) at (4.5,-2) {$\mathscr{R}(M\otimes \mathscr{U}\mathscr{B})$};

\node (i11) at (1.25,-0.5) {$\mathscr{R}M\otimes \mathscr{R}\mathscr{U}\mathscr{B}\otimes \mathscr{R}\mathscr{U}\mathscr{B}$};
\node (i21) at (3.25,-0.5) {$\mathscr{R}M\otimes \mathscr{R}\mathscr{U}(\mathscr{B}\otimes \mathscr{B})$};
\node (i12) at (1.25,-1.5) {$\mathscr{R}(M\otimes \mathscr{U}\mathscr{B})\otimes \mathscr{R}\mathscr{U}\mathscr{B}$};
\node (i22) at (2.5,-2.5) {$\mathscr{R}(M\otimes \mathscr{U}\mathscr{B}\otimes \mathscr{U}\mathscr{B})$};

\node (ii1) at (2.5,-1) {$\mathscr{R}M\otimes \mathscr{R}(\mathscr{U}\mathscr{B}\otimes \mathscr{U}\mathscr{B})$};
\node (ii2) at (3.5,-1.5) {$\mathscr{R}(M\otimes \mathscr{U}(\mathscr{B}\otimes \mathscr{B}))$};

\path[->,font=\scriptsize,>=angle 90]
(11) edge node[left] {$1\otimes \mathrm{g}\otimes 1$} (12)
(12) edge node[left] {$\mu^{(\mathscr{R})}\otimes 1$} (13)
(13) edge node[left] {$\mathscr{R}\rho\otimes 1$} (14)
(14) edge node[below] {$1\otimes \mathrm{g}$} (24)
(24) edge node[below] {$\mu^{(\mathscr{R})}$} (34)
(34) edge node[below] {$\mathscr{R}\rho$} (44);

\path[->,font=\scriptsize,>=angle 90]
(11) edge node[above] {$1\otimes\mu^{(\mathscr{B})}$} (41)
(41) edge node[right] {$1\otimes\mathrm{g}$} (42)
(42) edge node[right] {$\mu^{(\mathscr{R})}$} (43)
(43) edge node[right] {$\mathscr{R}\rho$} (44);

\path[->,font=\scriptsize,>=angle 90]
(11) edge node[above right] {$1\otimes \mathrm{g}\otimes \mathrm{g}$} (i11)
(12) edge node[above left] {$1\otimes 1\otimes\mathrm{g}$} (i11)
(i11) edge node[above] {$1\otimes \mu^{(\mathscr{R}\circ\mathscr{U})}$} (i21)
(i21) edge node [fill=white,rounded corners=1pt,inner sep=0.3pt] {$1\otimes\mathscr{R}\mathscr{U}\mu^{(\mathscr{B})}$} (42)
(13) edge node[above left] {$1\otimes \mathrm{g}$} (i12)
(i11) edge node[left] {$\mu^{(\mathscr{R})}\otimes 1$} (i12)
(i12) edge node[left] {$\mathscr{R}\rho \otimes 1$} (24);

\path[->,font=\scriptsize,>=angle 90]
(i12) edge node[below left] {$\mu^{(\mathscr{R})}$}(i22) 
(i22) edge node[right] {$\mathscr{R}(\rho \otimes 1)$} (34)
(i22) edge node[below right] {$\mathscr{R}(1\otimes \mu^{(\mathscr{U}\circ\mathscr{B})})$} (43)
;
\path[->,font=\scriptsize,>=angle 90]
(i11) edge node[below left] {$1\otimes\mu^{(\mathscr{R})}$}(ii1) 
(ii1) edge node[left] {$1\otimes \mathscr{R}\mu^{(\mathscr{U})}$} (i21)
(ii1) edge node[above left] {$\mu^{(\mathscr{R})}$} (i22)
(i21) edge node[right] {$\mu^{(\mathscr{R})}$}(ii2) 
(i22) edge node[fill=white,rounded corners=1pt,inner sep=0.3pt]{$\mathscr{R}(1\otimes \mu^{(\mathscr{U})})$} (ii2)
(ii2) edge node[below left] {$\mathscr{R}(1 \otimes \mathscr{U}\mu^\mathscr{B})$} (43)
;
\end{tikzpicture}
\end{equation}
In the above diagram, the top pentagon commutes since $\mathrm{g}$ is a functor; there are naturality squares for $\mu$; and the bottom right square is obtained by applying $\mathscr{R}$ to compatibility of $\rho$ with $\mu^{(\mathscr{U}\circ \mathscr{B})}$. Similarly, functoriality of $\mathrm{g}$ with respect to identities leads to compatibility of $\widetilde{\mathscr{R}}\rho$ with $\eta^{(\mathscr{B})}$.
Compatibility of $\widetilde{\mathscr{R}}\sigma$ with $\widetilde{\mathscr{R}}\rho$ (and $\widetilde{\mathscr{R}}\lambda$) follows directly from the compatibility of $\sigma$ with $\rho$ (and $\lambda$).

The components of the unit and counit of the local adjunctions are given by components of $\gamma$ and $\epsilon$:
\begin{align*}
\tilde{\eta}_\mathscr{A}^\mathscr{B}:1_{\mathscr{N}(\mathscr{A},\mathscr{B})}
	 \Rightarrow\widetilde{\mathscr{R}}_\mathscr{A}^\mathscr{B}(\mathscr{U}\circ -) \ , \
((\tilde{\eta}_\mathscr{A}^\mathscr{B})_N)_B^A:N_B^A
	\xrightarrow{\mathrm{g}_{N_B^A}}\mathscr{R}\mathscr{U}N_B^A\\
\tilde{\epsilon}_{\mathscr{A}}^{\mathscr{B}}:
	\mathscr{U}\circ\widetilde{\mathscr{R}}_\mathscr{A}^\mathscr{B}( -)
	 \Rightarrow
	1_{\mathscr{M}(\mathscr{U}\circ\mathscr{A},\mathscr{U}\circ\mathscr{B})} \ , \
((\tilde{\epsilon}_\mathscr{A}^\mathscr{B})_M)_B^A :
	\mathscr{U}\mathscr{R}
	M_B^A
	\xrightarrow{\epsilon_{M_B^A}}M_B^A\,.
\end{align*}
They form module morphisms, as proved by diagrams

\begin{equation*}
\begin{tikzpicture}[xscale=8, yscale=1.5, every node/.style={scale=1},baseline=(current  bounding  box.center)]
\node (urmub) at (0,0) {$\mathscr{U}\mathscr{R}M_B^A\otimes
	 \mathscr{U}\mathscr{B}_{B'}^B$};
\node (urmb) at (0,-1) {$\mathscr{U}(\mathscr{R}M_B^A\otimes\mathscr{B}_{B'}^B)$};
\node (urmrub) at (0,-2) {$\mathscr{U}(\mathscr{R}M_B^A\otimes
		\mathscr{R}\mathscr{U}\mathscr{B}_{B'}^B)$};
\node (urmub2) at (0,-3) {$\mathscr{U}\mathscr{R}(M_B^A\otimes
		\mathscr{U}\mathscr{B}_{B'}^B)$};
\node (urm) at (0,-4) {$\mathscr{U}\mathscr{R}M_{B'}^A$};

\node (mub) at (1,-1) {$M_B^A\otimes\mathscr{U}\mathscr{B}_{B'}^B$};
\node (m) at (1,-2) {$M_{B'}^A$};

\node (urmurub) at (0.5,-1.25) {$\mathscr{U}\mathscr{R}M_B^A\otimes
	 \mathscr{U}\mathscr{R}\mathscr{U}\mathscr{B}_{B'}^B$};
\path[->,font=\scriptsize,>=angle 90]
(urmub) edge node [left] {$\mu^{(\mathscr{U})}$} (urmb)
(urmb) edge node [left] {$\mathscr{U}(1\otimes \mathrm{g})$} (urmrub)
(urmrub) edge node [left] {$\mathscr{U}\mu^{(\mathscr{R})}$} (urmub2)
(urmub2) edge node [left] {$\mathscr{U}\mathscr{R}\rho^{(M)}$} (urm)
(urm) edge node [below] {$\epsilon$} (m)
(urmub) edge node [above] {$\epsilon\otimes 1$} (mub)
(mub) edge node [right] {$\rho^{(M)}$} (m)
(urmub) edge node [right] {$1\otimes \mathscr{U}\mathrm{g}$} (urmurub)
(urmurub) edge node [above left] {$\epsilon\otimes\epsilon$} (mub)
(urmurub) edge node [above] {$\mu^{(\mathscr{U})}$} (urmrub)
(urmurub) edge node [right] {$\mu^{(\mathscr{U}\mathscr{R})}$} (urmub2)
;
\path[->,font=\scriptsize,>=angle 90]
(urmub2) edge 
	node [below right] {$\epsilon$} (mub)
;

\end{tikzpicture}
\end{equation*}
\begin{equation*}
\begin{tikzpicture}[xscale=6, yscale=1.5, every node/.style={scale=1},baseline=(current  bounding  box.center)]
\node (nb) at (0,0) {$N_B^A\otimes \mathscr{B}_B^{B'}$};
\node (runb) at (1,0) {$\mathscr{R}\mathscr{U}N_B^A\otimes \mathscr{B}_{B'}^B$};
\node (n) at (0,-3) {$N_{B'}^A$};
\node (runrub) at (1,-1) 
	{$\mathscr{R}\mathscr{U}N_B^A\otimes 
		\mathscr{R}\mathscr{U}\mathscr{B}_{B'}^B$};
\node (runub) at (1,-2) 
	{$\mathscr{R}(\mathscr{U}N_B^A\otimes 
		\mathscr{U}\mathscr{B}_{B'}^B)$};
\node (run) at (1,-3) {$\mathscr{R}\mathscr{U}N_{B'}^A$};
\node (irunb) at (0.35,-1.5) {$\mathscr{R}\mathscr{U}(N_B^A\otimes \mathscr{B}_B^{B'})$};
\path[->,font=\scriptsize,>=angle 90]
(nb) edge node[left] {$\rho^{(N)}$} (n)
(n) edge node[below] {$\mathrm{g}$} (run)
(nb) edge node[above] {$\mathrm{g}\otimes 1$} (runb)
(nb) edge node[right] {$\mathrm{g}$} (irunb)
(runrub) edge node[above left] {$\mu^{(\mathscr{R}\mathscr{U})}$} (irunb)
(irunb) edge node[below left] {$\mathscr{R}\mathscr{U}\rho^{(N)}$} (run)
(runub) edge node[above right] {$\mathscr{R}\mu^{(\mathscr{U})}$} (irunb)
(runb) edge node[right] {$1 \otimes \mathrm{g}$} (runrub)
(nb) edge node[above right] {$\mathrm{g}\otimes \mathrm{g}$} (runrub)
(runrub) edge node[right] {$\mu^{(\mathscr{R})}$} (runub)
(runub) edge node[right] {$\mathscr{R}\rho^{(\mathscr{U}N)}$} (run)
;
\end{tikzpicture}
\end{equation*}
and they satisfy the adjunction axioms because $\mathrm{g}$ and $\epsilon$ do. 
Since $\widetilde{\mathscr{U}}$ has local right adjoints, it preserves local colimits, which, together with pseudofunctoriality of $\mathscr{U}$, gives sufficient conditions for pseudofunctoriality of $\widetilde{\mathscr{U}}$:
\begin{equation*}
\begin{tikzpicture}[scale=1.5, every node/.style={scale=0.9},baseline=(current  bounding  box.center)]
\node (A1) at (-3.5,0.7) {
$\sum\mathscr{U}M_{B'}^A\otimes\mathscr{U}\mathscr{B}_B^{B'}\otimes\mathscr{U}N_{C'}^B$};
\node (A2) at (0,0.7) {
$\sum\mathscr{U}M_{B}^A\otimes\mathscr{U}N_{C'}^B$};
\node (A3) at (2.7,0.7) {
$(\mathscr{U}N\circ_{\mathscr{U}\mathscr{B}}\mathscr{U}M)_{C'}^A$};
\node (B1) at (-3.5,-0.7) {
$\sum \mathscr{U}(M_{B'}^A\otimes\mathscr{B}_B^{B'}\otimes N_{C'}^B)$};
\node (B2) at (0,-0.7) {
$\sum \mathscr{U}(M_{B}^A\otimes N_{C'}^B)$};
\node (B3) at (2.7,-0.7) {
$\mathscr{U}(N\circ_{\mathscr{B}}M)_{C'}^A\,.$};
\path[transform canvas={yshift=1mm},->,font=\scriptsize,>=angle 90]
(A1) edge node[above] {$\mathscr{U}\rho\otimes 1$} (A2);
\path[transform canvas={yshift=-1mm},->,font=\scriptsize,>=angle 90]
(A1)edge node[below] {$1\otimes \mathscr{U}\lambda$} (A2);
\path[->,font=\scriptsize,>=angle 90]
(A2) edge node[above] {$\text{coeq}$} (A3);
\path[transform canvas={yshift=1mm},->,font=\scriptsize,>=angle 90]
(B1) edge node[above] {$\mathscr{U}(\rho\otimes 1)$} (B2);
\path[transform canvas={yshift=-1mm},->,font=\scriptsize,>=angle 90]
(B1) edge node[below] {$\mathscr{U}(1\otimes\lambda)$} (B2);
\path[->,font=\scriptsize,>=angle 90]
(B2) edge node[below] {$\mathscr{U}(\text{coeq})$} node[above] {$\text{coeq}$} (B3);
\path[transform canvas={xshift=19mm},->,font=\scriptsize,>=angle 90]
(A1) edge node[left] {$\sum \mu^{(\mathscr{U})}$} (B1);
\path[transform canvas={xshift=-13mm},->,font=\scriptsize,>=angle 90]
(A2) edge node[right] {$\sum \mu^{(\mathscr{U})}$} (B2);
\path[dashed,->,font=\scriptsize,>=angle 90]
(A3) edge node[right] {$\mu^{(\widetilde{\mathscr{U}})}$} (B3);
\end{tikzpicture}
\end{equation*}
Since $\widetilde{\mathscr{U}}$ is a pseudofunctor and has local right adjoints, by Proposition 2.7 of \cite{Kelly2002}
(see Remark~\ref{adjinCaten}), $\widetilde{\mathscr{R}}$ extends to a 2-sided enriched category which is a right adjoint to $\widetilde{\mathscr{U}}$.
\end{proof}

Explicitly, the unit for $\widetilde{\mathscr{R}}$ is a module morphism defined using the enriched functor 
\begin{align*}
\eta_{\mathscr{A}}^{(\widetilde{\mathscr{R}})}:=(\mathrm{g}\circ \mathscr{A})_* \ ,
\end{align*}
and the multiplication components
\begin{equation*}
\widetilde{\mathscr{R}}(N)\circ_\mathscr{B}\widetilde{\mathscr{R}}(M)
\xRightarrow{\mu^{(\widetilde{\mathscr{R}})}}
\widetilde{\mathscr{R}}(N\circ_{\mathscr{U}\mathscr{B}}M)
\end{equation*}
are given by the right column of
\begin{equation*}
\hspace{-2cm}
\begin{tikzpicture}[scale=1.5, every node/.style={scale=1},baseline=(current  bounding  box.center)]
\node (A1) at (-3.5,0.7) {
$\sum\mathscr{R}M_{B'}^A\otimes\mathscr{B}_B^{B'}\otimes\mathscr{R}N_{C'}^B$};
\node (A2) at (0,0.7) {
$\sum\mathscr{R}M_{B}^A\otimes\mathscr{R}N_{C'}^B$};
\node (A3) at (2.7,0.7) {
$(\mathscr{R}N\circ_{\mathscr{B}}\mathscr{R}M)_{C'}^A$};
\node (B1) at (-3.5,-0.7) {
$\mathscr{R}\sum M_{B'}^A\otimes\mathscr{U}\mathscr{B}_B^{B'}\otimes N_{C'}^B$};
\node (B2) at (0,-0.7) {
$\mathscr{R}\sum M_{B}^A\otimes N_{C'}^B$};
\node (B3) at (2.7,-0.7) {
$\mathscr{R}(N\circ_{\mathscr{U}\mathscr{B}}M)_{C'}^A$};-
\path[transform canvas={yshift=1mm},->,font=\scriptsize,>=angle 90]
(A1) edge node[above] {$\widetilde{\mathscr{R}}\rho\otimes 1$} (A2);
\path[transform canvas={yshift=-1mm},->,font=\scriptsize,>=angle 90]
(A1)edge node[below] {$1\otimes \widetilde{\mathscr{R}}\lambda$} (A2);
\path[->,font=\scriptsize,>=angle 90]
(A2) edge node[above] {$\text{coeq}$} (A3);
\path[transform canvas={yshift=1mm},->,font=\scriptsize,>=angle 90]
(B1) edge node[above] {$\mathscr{R}(\rho\otimes 1)$} (B2);
\path[transform canvas={yshift=-1mm},->,font=\scriptsize,>=angle 90]
(B1) edge node[below] {$\mathscr{R}(1\otimes\lambda)$} (B2);
\path[->,font=\scriptsize,>=angle 90]
(B2) edge node[below] {$\mathscr{R}(\text{coeq})$} (B3);
\path[transform canvas={xshift=19mm},->,font=\scriptsize,>=angle 90]
(A1) edge node[left] {$(\mathscr{R}(i_{BB'}) \ \mu^{(\mathscr{R})} \ (1\otimes\mathrm{g}\otimes 1))_{BB'}$} (B1);
\path[transform canvas={xshift=-13mm},->,font=\scriptsize,>=angle 90]
(A2) edge node[right] {$(\mathscr{R}(i_{B}) \  \mu^{(\mathscr{R})})_{B}$} (B2);
\path[dashed,->,font=\scriptsize,>=angle 90]
(A3) edge node[right] {$\mu^{(\widetilde{\mathscr{R}})}$} (B3);
\end{tikzpicture}
\end{equation*}
where the top line is the coequalizer defining composition of those modules in $\mathscr{N}$, the bottom line is the result of applying $\mathscr{R}$ to the defining coequalizer for composition of those modules in $\mathscr{M}$, morphisms $i_{B}$ and $i_{BB'}$ are the coproduct inclusions, and $(-)_B$ denotes the induced map for mapping out of a coproduct.

\begin{remark}\label{monoidcaseadjoint}
Following Remark~\ref{monoidcase}, we point out that Theorem~\ref{thm:adjointModU} restricts to the one-object case.
In particular, if $\CC$, $\CD$ and $\CE$ are monoidal categories and both $\CC$ and $\CD$ admit reflexive coequalizers
preserved by tensoring with an object then each adjunction $\CU\dashv \CR$ with $\CU : \CD\to \CC$ a strong monoidal
functor induces an adjunction $\widetilde{\CU}\dashv \widetilde{\CR}$ in $\mathrm{Caten}$ where
$\widetilde{\CU}$ is the pseudofunctor 
$$\mathrm{Mod1}(\CE,\CU) : \mathrm{Mod1}(\CE,\CD)\lra \mathrm{Mod1}(\CE,\CC) \ .$$  
In particular, when $\CE = \mathbf{1}$, we obtain a right adjoint in $\mathrm{Caten}$ for
$\widetilde{\CU} : \mathrm{Mod}(\CD)\lra \mathrm{Mod}(\CC)$.
\end{remark}

\section{Pseudocomonads in $\text{Caten}$ and their coalgebras}\label{sec:comonads}

Pseudocomonads in a tricategory are dual to pseudomonads; see \cite{60, Lack2000}.
However, we will be more explicit in the case of interest here.

Let $\mathscr{G}:\mathscr{V}\rightarrow\mathscr{V}$ be a pseudocomonad in $\text{Caten}$; that is, a 2-sided enriched category with enriched functors
\begin{align*}
1_\mathscr{V}\xleftarrow{\mathrm{e}}\mathscr{G}\xrightarrow{\mathrm{d}}\mathscr{G}^2
\end{align*}
and three enriched natural isomorphisms 
\begin{eqnarray*}
\xymatrix{
\CG \ar[d]_{\mathrm{d}}^(0.5){\phantom{aaa}}="1" \ar[rr]^{\mathrm{d}}  && \CG^2 \ar[d]^{\CG \mathrm{d}}_(0.5){\phantom{aaa}}="2" \ar@{=>}"1";"2"^-{\cong \ \alpha}
\\
\CG^2 \ar[rr]_-{\mathrm{d}\CG} && \CG^3 
}
\qquad
\xymatrix{
& \CG \ar[ld]_{\mathrm{d}}^(0.5){\phantom{a}}="1"   \ar[rd]^{1_{\CG}}_(0.5){\phantom{a}}="2" \ar@{=>}"1";"2"^-{\cong \ \lambda}
\\
\CG^2 \ar[rr]_-{\mathrm{e}\CG} && \CG 
}
\qquad
\xymatrix{
& \CG \ar[ld]_{\mathrm{d}}^(0.5){\phantom{a}}="1"   \ar[rd]^{1_{\CG}}_(0.5){\phantom{a}}="2" \ar@{<=}"1";"2"^-{\cong \ \rho}
\\
\CG^2 \ar[rr]_-{\CG \mathrm{e}} && \CG 
}
\end{eqnarray*}
satisfying two axioms. The existence of the span morphism $\text{Ob}\mathscr{G}\xrightarrow{\text{Ob}\mathrm{e}} \text{Ob}\mathscr{V}$ forces $$G_+=G_-=(\text{Ob}\mathrm{e}) (G)=:G_0 \ ,$$ for all objects $G$. To avoid extra brackets, we are taking the objects of $\CG^3$ to be triples 
$(G,H,K)$ of objects of $\CG$ such that $G_0=H_0=K_0$. The two counit constraints $\lambda$ and $\rho$ give isomorphisms $$(\text{Ob}\mathrm{d})(G)\xrightarrow{(\lambda_G,\rho_G^{-1})}(G,G)$$ which allow $\mathrm{d}$ to be replaced, up to isomorphism, by a comultiplication given by the diagonal function on objects, and whereby $\alpha$, $\lambda$ and $\rho$ become identities. Then $\CG$ is a comonad on $\CV$ in the bicategory
obtained from $\mathrm{Caten}$ by ignoring 3-cells. We will henceforth assume we have made this replacement and simply refer to $\CG$ as a ``comonad'' in $\mathrm{Caten}$. 
Then the remaining data for $\mathscr{G}$ are given by endofunctors
\begin{align*}
\mathscr{G}_G^{G'}:\mathscr{V}_{G_0}^{G'_0}\rightarrow \mathscr{V}_{G_0}^{G'_0}
\end{align*}
and natural transformations with components
\begin{align*}
\mu_{v',v} = (\mu_{G,G''}^{G'})_{v',v}&:\mathscr{G}_{G'}^{G''}(v')\otimes \mathscr{G}_G^{G'}(v)\rightarrow \mathscr{G}_G^{G''}(v'\otimes v)\\
\eta = \eta_G&:1_{G_0}\rightarrow \mathscr{G}_G^G(1_{G_0})\\
\mathrm{d}_v = (\mathrm{d}_G^{G'})_v&:\mathscr{G}_G^{G'}(v)\rightarrow (\mathscr{G}_G^{G'})^2(v)\\
\mathrm{e}_v = (\mathrm{e}_G^{G'})_v&:\mathscr{G}_G^{G'}(v)\rightarrow v
\end{align*}
satisfying enriched functor compatibility axioms, which together with the comonad axioms, are the monoidal comonad axioms dual to the opmonoidal monad axioms appearing in \cite{Bruguieres2011}.
In particular, each $(\mathscr{G}_G^{G'}, \mathrm{e}_G^{G'}, \mathrm{d}_G^{G'})$ is a comonad on $\mathscr{V}_{G_0}^{G'_0}$ in the usual sense (in $\text{Cat}$). 

We will proceed to define a bicategory $\mathscr{V}^\mathscr{G}$ with the same objects as $\mathscr{G}$ and with homs the categories of Eilenberg-Moore-coalgebras
\begin{equation*}
\mathscr{V}^\mathscr{G}(G,G'):=\mathscr{V}(G_0,G'_0)^{\mathscr{G}(G,G')}\,.
\end{equation*}
The identity coalgebra is $(1_{G_0},\eta_G)$ and composition is given on coalgebras by
\begin{align}
\mathscr{V}^\mathscr{G}(G',G'')\times \mathscr{V}^\mathscr{G}(G,G')&\rightarrow
\mathscr{V}^\mathscr{G}(G,G'')\label{eq:tensorCoalg}\\
(v',\gamma_{v'}),(v,\gamma_{v})&\mapsto
(v'\otimes v, (\mu_{GG''}^{G'})_{v',v}\circ (\gamma_{v'}\otimes \gamma_v))\,.\nonumber
\end{align}
The assigned pair is indeed a coalgebra: compatibility with $\mathrm{d}$ is proved by\footnote{When indices are omitted they can be deduced from the context. For example, $\mathscr{G}_G^{G'}(v)$ is the full notation for $\mathscr{G}v$.}
\begin{equation*}
\begin{tikzpicture}[xscale=4,yscale=1.3, every node/.style={scale=1},baseline=(current  bounding  box.center)]
\node (M1) at (0,-1) {$v'\otimes v$};
\node (TM1) at (0,-2) {$\mathscr{G}v' \otimes \mathscr{G}v$};
\node (TE1) at (0,-3) {$\mathscr{G}(v' \otimes v)$};

\node (23) at (1,-3) {$\mathscr{G}(\mathscr{G}v' \otimes \mathscr{G}v)$};
\node (TTE) at (2,-3) {$\mathscr{G}^2(v' \otimes v)$};

\node (TM3) at (2,-1) {$\mathscr{G}(v' \otimes v)$};

\node (N) at (1,-1) {$\mathscr{G}v' \otimes \mathscr{G}v$};
\node (TN) at (1,-2) {$\mathscr{G}^2 v' \otimes \mathscr{G}^2 v$};

\path[->,font=\scriptsize,>=angle 90]
(M1) edge node[left] {$\gamma\otimes\gamma$} (TM1)
(TM1) edge node[left] {$\mu$} (TE1)
(TE1) edge node[below] {$\mathscr{G}(\gamma\otimes \gamma)$} (23)
(23) edge node[below] {$\mathscr{G}\mu$} (TTE);

\path[->,font=\scriptsize,>=angle 90]
(TM3) edge node[right] {$\mathrm{d}$} (TTE);

\path[->,font=\scriptsize,>=angle 90]
(N) edge node[above] {$\mu$} (TM3)
(M1) edge node[above] {$\gamma\otimes\gamma$} (N)
(TM1) edge node[below] {$\mathscr{G}\gamma\otimes\mathscr{G}\gamma$} (TN)
(N) edge node[left] {$\mathrm{d}\otimes \mathrm{d}$} (TN)
(TN) edge node[left] {$\mu$} (23)
(TN) edge node[above] {$\mu^{(\mathscr{G}^2)}$} (TTE);

\end{tikzpicture}
\end{equation*}
where the upper left square is a componentwise compatibility of local coalgebras $\gamma$ with comultiplication, the bottom left square is naturality of $\mu$, the triangle is the definition of composition for the composite category, and the remaining square is compatibility of the enriched functor $d$ with compositions in its source and target (a typical bimonoid axiom). Similarly, compatibility of \eqref{eq:tensorCoalg} with $\mathrm{e}$ follows from its being an enriched functor. 
The assignment extends to coalgebra morphisms since $\mu$ is natural. The unit and associativity isomorphisms are inherited from $\mathscr{V}$; they are coalgebra morphisms and satisfy the monoidale axioms because they do in $\mathscr{V}$.

There is an underlying (strict) functor
$\mathscr{U}:\mathscr{V}^\mathscr{G}
\rightarrow\mathscr{V}$
which sends $G$ to the underlying object $G_0$ in $\mathscr{V}$ and disregards the coalgebra structure on 1-cells. By construction, each $\mathscr{U}_V^{V'}$ has a right adjoint $\mathscr{R}_V^{V'}$, and by Theorem 2.7 of \cite{Kelly2002} the right adjoints are part of a 2-sided enriched category $\mathscr{R}:\mathscr{V}\rightarrow \mathscr{V}^\mathscr{G}$ with the same objects as $\mathscr{G}$, with span legs given by $G_-=G_0$ and $G_+=G$, with multiplication having components
\begin{align*}
\lefteqn{\mathscr{R}_{G'}^{G''}(v')\otimes \mathscr{R}_G^{G'}(v)\xrightarrow{(\mu_{GG''}^{G'})_{v',v}} \mathscr{R}_{G}^{G''}(v'\otimes v)} \\
& : = \ (\mathscr{G}v'\otimes \mathscr{G}v,\mu_{v',v}\circ (\mathrm{d}_{v'}\otimes \mathrm{d}_v))\xrightarrow{\mu_{v',v}}
(\mathscr{G}(v'\otimes v),\mathrm{d}_{v'\otimes v})
\end{align*}
and unit having components
\begin{align*}
1_{G} \xrightarrow{\eta_G} \mathscr{R}_G^G(1_{G_0}) \ 
: = \ (1_{G_0},\eta_G)\xrightarrow {\eta_G} (\mathscr{G}1_{G_0},\mathrm{d}_{1_{G_0}}) \,.
\end{align*}

Now we have an adjunction in $\mathrm{Caten}$.
\begin{equation}\label{diag:catenAdj}
\begin{tikzpicture}[scale=1.5, every node/.style={scale=1},baseline=(current  bounding  box.center)]
\node (A1) at (-1,0) {$\mathscr{V}^\mathscr{G}$};
\node (A2) at (1,0) {$\mathscr{V}$};
\node (ad) at (0,0) {$\bot$};
\path[,->,font=\scriptsize,>=angle 90,bend left]
(A1) edge node[above] {$\mathscr{U}$} (A2);
\path[<-,font=\scriptsize,>=angle 90,bend right]
(A1)edge node[below] {$\mathscr{R}$} (A2);
\end{tikzpicture}
\end{equation}
The counit and the unit of the adjunction are given by the enriched functors
\begin{align}\label{eq:gammaFunct}
\mathscr{U}\circ\mathscr{R}=\mathscr{G}&\xrightarrow{\mathrm{e}} 1_\mathscr{V} \ \text{ and }
1_{\mathscr{V}^\mathscr{G}}\xrightarrow{\mathrm{g}} \mathscr{R}\circ\mathscr{U}
\end{align}
where $(\text{ob}\mathrm{g})(G)=(G,G)$ and  
$\mathrm{g}_{(v,\gamma_v)} = \gamma_v : (v,\gamma_v) \rightarrow (\mathscr{G}_G^{G'}v,\mathrm{d}_v)$.

\begin{example} \label{comonoidalenrichedcats}
Let $\CC$ be a braided monoidal category. Then the 2-category $\CC\text{-}\mathrm{Cat}$
of $\CC$-enriched categories, enriched functors and natural transformations is monoidal; the symmetric case
was already in \cite{EilKel1966} and the braided in \cite{Joyal1993}. A {\em comonoidal} $\CC$-category $\CA$
is a monoidale (= pseudomonoid) in $\CC\text{-}\mathrm{Cat}^{\mathrm{op}}$. As pointed out in Section 8 of
\cite{60}, the cotensor $\CC$-functor $\mathrm{D} : \CA\to \CA\otimes \CA$ can be taken on objects to be the 
diagonal function $\mathrm{D}A = (A,A)$; then, on homs, $\mathrm{D}$ and the counit $\mathrm{E}$ provide morphisms
\begin{eqnarray*}
\mathrm{D}_A^{A'} : \CA(A,A')\to \CA(A,A')\ox \CA(A,A') \ \text{  and  } \ \mathrm{E}_A^{A'} : \CA(A,A')\to I
\end{eqnarray*}
 in $\CA$. These are equipped with coassociativity and counital constraints satisfying the two axioms for
 a comonoidale. Consequently, we obtain, as follows, a pseudocomonad $\CG$ on the bicategory 
 $\CV = \Sigma \CC$ with one object $o$ and hom category $\CV(o,o) = \CC$. The objects of $\CG$ are
 those of $\CA$. The endofunctor $\CG_A^{A'} : \CV^o_o\to \CV_o^o$ is $\CA(A,A')\ox - : \CC\to \CC$.
 Then, using the braiding on $\CC$ and composition in $\CA$, we have the composite 
 \begin{eqnarray*}
\CA(A',A'')\ox X'\ox \CA(A,A')\ox X \cong \CA(A',A'')\ox \CA(A,A') \ox X' \ox X \to \CA(A,A'')\ox X'\ox X 
\end{eqnarray*}
defining $\mu_{X'}^{X''}$ while $\eta : I\to \CA(A,A)\cong \CA(A,A)\ox I$ uses the identity structure of $\CA$
and the right unit constraint of $\CC$. The morphisms $\mathrm{d}_X : \CA(A,A')\ox X\to \CA(A,A')\ox \CA(A,A')\ox X$ and $\mathrm{e}_X : \CA(A,A')\ox X\to X$ are defined by tensoring with $\mathrm{D}_A^{A'}$
and $\mathrm{E}_A^{A'}$.

The bicategory $\CV^{\CG}$ has the same objects as $\CA$; so it is not $\Sigma$ of a monoidal category unless $\CA$ is a bimonoid in $\CC$. The hom category $\CV^{\CG(A,A')} = \CC^{\CA(A,A')\ox -}$ is the
category of coalgebras for the comonad $\CA(A,A')\ox -$ on $\CC$. For composition in the bicategory, we refer
the reader back to \eqref{eq:tensorCoalg}.  
\end{example}
\bigskip

Now we present a version of Beck's comonadicity theorem (for example, see \cite{Beck1967, Barr1985} where ``triple'' means ``monad'') for use in the rest of the chapter.
\begin{proposition}\label{prop:comonadic}
Any 2-sided enriched category $\mathscr{L}:\mathscr{W}\rightarrow \mathscr{V}$ such that
\begin{itemize}
\item[(a)] $\mathscr{L}$ has a right adjoint $\mathscr{R}$ in $\mathrm{Caten}$, 
\item[(b)] $\mathscr{L}$ is locally conservative, and
\item[(c)] $\mathscr{W}$ has, and $\mathscr{L}$ preserves, local $\mathscr{L}$-split equalizers,
\end{itemize}
determines an identity-on-objects biequivalence $\mathscr{W}\sim\mathscr{V}^{\mathscr{G}}$ over $\CV$, where $\mathscr{G}=\mathscr{L}\circ \mathscr{R}$ is the generated comonad.
\end{proposition}
\begin{proof}
By Proposition 2.7 of \cite{Kelly2002}, $\mathscr{L}$ has a right adjoint if and only if it is a pseudofunctor and each functor $\mathscr{L}(L,L')$ has a right adjoint which we denote by $\mathscr{R}(L,L')$. Then, the right adjoint $\mathscr{R}$ has the same objects as $\mathscr{L}$ (and $\mathscr{W}$, since $\mathscr{L}$ is a pseudofunctor), and homs are precisely the $\mathscr{R}(L,L')$. From the usual Beck comonadicity theorem it follows that $\mathscr{W}$and $\mathscr{V}^{\mathscr{G}}$ have
compatibly equivalent homs; since they have the same objects, we obtain the desired biequivalence.
\end{proof}

Return now to our comonad $\CG$ on $\CV$ in $\text{Caten}$. 
The category $\text{Caten}(\mathscr{X},\mathscr{V})$  has an induced comonad $\text{Caten}(\mathscr{X},\mathscr{G})$ on it. In particular, when $\mathscr{X}=\mathscr{V}^\mathscr{G}$ there is a natural coalgebra structure on $\mathscr{U}$ given by the enriched functor
\begin{equation*}
\mathscr{U}\xrightarrow{\mathscr{U}\circ\mathrm{g}}	\mathscr{G}\circ \mathscr{U}
\end{equation*}
whose effect on homs has components the coactions $\gamma_v:v\rightarrow \mathscr{G}v$.

\begin{lemma}\label{lem:whisk}
Let $\mathscr{X}$ be a bicategory. The functor
\begin{equation*}
\mathrm{Moden}(\mathscr{X},\mathscr{V}^\mathscr{G})(\mathscr{A},\mathscr{B})
\xrightarrow{\mathscr{U}\circ-}
\mathrm{Moden}(\mathscr{X},\mathscr{V})(\mathscr{U}\circ\mathscr{A},\mathscr{U}\circ\mathscr{B})
\end{equation*}
is conservative. Moreover the source has, and the functor preserves,  $(\mathscr{U}\circ-)\text{-}$split equalizers.
\end{lemma}
\begin{proof}
Take modules $M,N:\mathscr{A}\xmrightarrow{}\mathscr{B}$. A module morphism $\sigma:M\Rightarrow N$ has components
\begin{equation*}
(\sigma_B^A)_x:M_{B}^{A}(x)\rightarrow N_{B}^{A}(x)
\end{equation*}
which are 2-cells in $\mathscr{V}^\mathscr{G}$, natural in $x\in \mathscr{X}_{B_-}^{A_-}$.

To prove $\mathscr{U}\circ -$ is conservative, denote by $\psi:\mathscr{U}\circ M\Rightarrow \mathscr{U}\circ N$ the inverse of $\mathscr{U}\circ \sigma$.
This precisely means that the component
\begin{equation*}
(\psi_B^A)_x:\mathscr{U}N_{B}^{A}(x)\rightarrow \mathscr{U}M_{B}^{A}(x)
\end{equation*}
is an inverse of the component 2-cell $(\sigma_B^A)_x$ in $\mathscr{V}$. Since $\mathscr{U}$ is locally conservative, $(\psi_B^A)_x$ is also a coalgebra morphism. Hence, naturality squares for $\psi_B^A$ consist of the same arrows regardless of whether it is seen as a morphism from $\mathscr{U}N_{B}^{A}$ to $\mathscr{U}M_{B}^{A}$, or from $N_{B}^{A}$ to $M_{B}^{A}$. Compatibility of $\psi$ with actions for $M$ and $N$ follows from the same compatibility conditions for $\sigma$ and the fact that they are inverse to each other.

Consider a pair $\sigma,\chi:M\Rightarrow N$ with a split equalizer
\begin{equation*}
\begin{tikzpicture}[xscale=0.8, every node/.style={scale=1},baseline=(current  bounding  box.center)]
\node (A1) at (3.5,0.7) {$\mathscr{U}\circ N$};
\node (A2) at (0,0.7) {$\mathscr{U}\circ M$};
\node (A3) at (-2.7,0.7) {$E$};
\path[transform canvas={yshift=1.7mm},<-,font=\scriptsize,>=angle 90]
(A1) edge [thick] node[transform canvas={yshift=-1.3mm},above] {$\mathscr{U}\circ \sigma$} (A2);
\path[transform canvas={yshift=-1.7mm},<-,font=\scriptsize,>=angle 90]
(A1)edge [thick] node[transform canvas={yshift=-1.3mm},above] {$\mathscr{U}\circ \chi$} (A2);
\path[transform canvas={yshift=0mm},->,font=\scriptsize,>=angle 90]
(A1)edge [bend left] node[transform canvas={yshift=-1mm},above] {$\psi$} (A2);
\path[transform canvas={yshift=0mm},<-,font=\scriptsize,>=angle 90]
(A2) edge [thick] node[transform canvas={yshift=-1mm},above] {$\xi$} (A3);
\path[transform canvas={yshift=0mm},->,font=\scriptsize,>=angle 90]
(A2)edge [bend left] node[transform canvas={yshift=-1mm},above] {$\phi$} (A3);
\end{tikzpicture}
\end{equation*}
meaning that we have the following componentwise formulas:
\begin{align*}
(\phi_B^A)_x (\xi_B^A)_x =1_{E_B^A(x)} \ , \ (\psi_B^A)_x  (\sigma_B^A)_x &= 1_{\mathscr{U}M_B^A(x)} \ , \ (\psi_B^A)_x  (\chi_B^A)_x = (\xi_B^A)_x  (\phi_B^A)_x\,.
\end{align*}
This in particular means that the pair $(\sigma_B^A)_x,(\chi_B^A)_x:M_B^A(x)\rightarrow N_B^A(x)$
has a $\mathscr{U}_{B_+}^{A_+}\text{-}$split equalizer in $\mathscr{V}_{B_{+0}}^{A_{+0}}$. Since $\mathscr{U}_{B_+}^{A_+}$ is comonadic, $(\xi_B^A)_x$ is an equalizer of $(\sigma_B^A)_x$ and $(\chi_B^A)_x$ in $\mathscr{V}_{B_{+}}^{A_{+}}$, with an algebra structure $\gamma_{E_B^A(x)}:=$ 
\begin{align*}
E_B^A(x) \xrightarrow{(\xi_B^A)_x}\mathscr{U}M_B^A(x)\xrightarrow{\gamma_{M_B^A(x)}}
	\mathscr{G}_{B_+}^{A_+}\mathscr{U}M_B^A(x)\xrightarrow{\mathscr{G}_{B_+}^{A_+}(\phi_B^A)_x}\mathscr{G}_{B_+}^{A_+} E_B^A(x)\,
\end{align*}
on its source.
The action components for the module $E$ are coalgebra morphisms: the proof for left action (dually for right) comes from the diagram (\ref{diag:actEcoalgMorph}) (all indices can be deduced from the top left term).
\begin{equation}\label{diag:actEcoalgMorph}
\begin{tikzpicture}[xscale=1.8,yscale=1.7, every node/.style={scale=1},baseline=(current  bounding  box.center)]
\node (E) at (0,0) {$\mathscr{A}_A^{A'}(x') \otimes E_B^A(x)$};
\node (M1) at (0,-1) {$\mathscr{G}\mathscr{A} \otimes M$};
\node (TM1) at (0,-2) {$\mathscr{G}\mathscr{A} \otimes \mathscr{G}M$};
\node (TE1) at (0,-3) {$\mathscr{G}\mathscr{A} \otimes\mathscr{G}E$};

\node (23) at (2.25,-3) {$G(\mathscr{A}\otimes E)$};
\node (TTE) at (4.5,-3) {$\mathscr{G}E$};

\node (M2) at (4.5,0) {$E$};
\node (TM3) at (4.5,-1) {$M$};
\node (TE2) at (4.5,-2) {$\mathscr{G}M$};

\node (N) at (2.25,-1) {$\mathscr{A}\otimes M$};
\node (TN) at (2.25,-2) {$\mathscr{G}(\mathscr{A}\otimes M)$};

\path[->,font=\scriptsize,>=angle 90]
(E) edge node[left] {$\mathrm{g}\otimes\xi$} (M1)
(M1) edge node[left] {$1\otimes\mathrm{g}$} (TM1)
(TM1) edge node[left] {$1\otimes\mathscr{G}\phi$} (TE1)
(TE1) edge node[below] {$\mu$} (23)
(23) edge node[below] {$\mathscr{G}\lambda$} (TTE);

\path[->,font=\scriptsize,>=angle 90]
(E) edge node[above] {$\lambda$} (M2)
(M2) edge node[right] {$\xi$} (TM3)
(TM3) edge node[right] {$\mathrm{g}$} (TE2)
(TE2) edge node[right] {$\mathscr{G}\phi$} (TTE);

\path[->,font=\scriptsize,>=angle 90]
(E) edge node[above] {$1\otimes \xi$} (N)
(N) edge node[above] {$\lambda$} (TM3)
(N) edge node[above] {$\mathrm{g}\otimes\mathrm{g}$} (TM1)
(TM1) edge node[below] {$\mu$} (TN)
(TN) edge node[above] {$\mathscr{G}\lambda$} (TE2)
(N) edge node[left] {$\mathrm{g}$} (TN)
(TN) edge node[left] {$\mathscr{G}(1\otimes \phi)$} (23);

\end{tikzpicture}
\end{equation}
Diagrams for compatibility of actions of $E$ with units and multiplications in $\mathscr{A}$ and $\mathscr{B}$ are the same as the ones for $\mathscr{U}\circ \mathscr{A}$ and $\mathscr{U}\circ \mathscr{B}$. This proves that $E:\mathscr{A}\xmrightarrow{} \mathscr{B}$ is a module. Components $(\xi_B^A)_x$ are natural in $x$, and compatible with actions of $E$ as coalgebra morphisms because they are natural and compatible as usual arrows. This proves that $\xi$ is a module morphism between $E$ (with coalgebra structure) and $M$.

It remains to show that $\xi$ is an equalizer of $\sigma$ and $\chi$, so assume $L\xrightarrow{\omega}M$ is another $\mathscr{V}^\mathscr{G}$-module morphism satisfying $\sigma \ \omega=\chi \ \omega$. The components of the composite $\phi \ \omega$ are coalgebra maps since $\mathscr{U}$ is locally comonadic, and naturality in $x$ and compatibility with actions follows as for $\xi$.
\end{proof}
\begin{corollary}\label{cor:whiskU}
The functor
\begin{equation*}
\mathrm{Caten}(\mathscr{X},\mathscr{V}^\mathscr{G})(\mathscr{A},\mathscr{B})
\xrightarrow{\mathscr{U}\circ-}
\mathrm{Caten}(\mathscr{X},\mathscr{V})(\mathscr{U}\circ\mathscr{A},\mathscr{U}\circ\mathscr{B})
\end{equation*}
is conservative. Moreover the source has, and the functor preserves,  $(\mathscr{U}\circ-)\text{-}$split equalizers.
\end{corollary}
\begin{proof}
This is a direct consequence of Proposition \ref{prop:2-sidY}, Lemma \ref{lem:whisk}, and commutativity of $(-)_*$ with $\mathscr{U}\circ-$.
\end{proof}
\begin{proposition}
The bicategory $\mathscr{V}^{\mathscr{G}}$ is the object of Eilenberg-Moore $\mathscr{G}$-coalgebras, in the sense of \cite{Street1972a}, for the comonad $\mathscr{G}$ in $\mathrm{Caten}$.
\end{proposition}
\begin{proof}
 Mapping out of $\mathscr{X}$ is a pseudofunctor
$\mathrm{Caten}(\mathscr{X},-):\mathrm{Caten}\rightarrow 2\text{-}\mathrm{CAT}$
 and therefore preserves adjunctions. In particular, applying it to (\ref{diag:catenAdj}) gives
\begin{equation*}
\begin{tikzpicture}[scale=1, every node/.style={scale=1},baseline=(current  bounding  box.center)]
\node (A1) at (-2,0) {$\mathrm{Caten}(\mathscr{X},\mathscr{V}^\mathscr{G})$};
\node (A2) at (2,0) {$\mathrm{Caten}(\mathscr{X},\mathscr{V})\,.$};
\node (ad) at (0,0) {$\bot$};
\path[,->,font=\scriptsize,>=angle 90,bend left]
(A1) edge node[above] {$\mathscr{U}':=\mathrm{Caten}(\mathscr{X},\mathscr{U})$} (A2);
\path[<-,font=\scriptsize,>=angle 90,bend right]
(A1)edge node[below] {$\mathscr{R}':=\mathrm{Caten}(\mathscr{X},\mathscr{R})$} (A2);
%
\end{tikzpicture}
\end{equation*}
The composite is isomorphic to $\mathrm{Caten}(\mathscr{X},\mathscr{G})$, and what remains to show is that $\mathscr{U}'$ is comonadic in the sense of Proposition \ref{prop:comonadic}. It has a right adjoint $\mathscr{R}'$, and the rest follows from Corollary \ref{cor:whiskU}. 
\end{proof}
\begin{proposition}\label{Convcoalg}
\begin{eqnarray*}
\mathrm{Conv}(\CX,\CV^{\CG})\cong \mathrm{Conv}(\CX,\CV)^{\mathrm{Conv}(\CX,\CG)} 
\end{eqnarray*}

\end{proposition}
\begin{proof}
We have
\begin{align*}
\mathrm{obConv}(\CX,\CV^{\CG}) & = \mathrm{ob}\CX \times \mathrm{ob}\CV^{\CG} \\
& = \mathrm{ob}\CX \times \mathrm{ob}\CG \\
& = \mathrm{obConv}(\CX, \CG) \\
& = \mathrm{obConv}(\CX,\CV)^{\mathrm{Conv}(\CX,\CG)} \ .
\end{align*}
Furthermore,
\begin{align*}
\mathrm{Conv}(\CX,\CV^{\CG})((X,G),(X',G')) & = [\CX(X,X'),\CV^{\CG}(G,G')] \\
& = [\CX(X,X'),\CV(G_0,G'_0)^{\CG(G,G')}] \\
& \cong [\CX(X,X'),\CV(G_0,G'_0)]^{[\CX(X,X'),\CG(G,G')]} \\
& = \mathrm{Conv}(\CX,\CV)^{\mathrm{Conv}(\CX,\CG)}((X,G),(X',G')) \ .
\end{align*}
Finally recall that the underlying functor for a category of Eilenberg-Moore coalgebras
creates colimits, so horizontal composition using the convolution formula \eqref{convolutioncoend}
respects the above identifications. 
\end{proof}

\section{Spans in categories of coalgebras}\label{sec:Spans}

We prove an adjunction result for bicategories $\mathrm{Spn}(\CE)$ of spans on a categories $\CE$ with pullbacks; see \cite{Ben1967} much like Theorem~\ref{thm:adjointModU} for the bicategories $\CV\text{-}\mathrm{Mod}$.

Then we provide a class of examples of comonads on bicategories $\mathrm{Spn}(\CE)$ of the kind in Section~\ref{sec:comonads}. This begins with any pullback preserving comonad $G$ on $\CE$. 
In particular, $\CE$ could be a topos, in which case the category $\CE^{G}$
of $G$-coalgebras is also a topos (for example, see \cite{Johnst2002} Section A, Remark 4.2.3). 
Another motivation and application lies in
the connection between categories enriched in bicategories of the form $\mathrm{Spn}(\CE)$
and locally small categories parametrized (or indexed) over $\CE$; see \cite{88}. 

The objects of the bicategory $\mathrm{Spn}(\CE)$ are those of $\CE$. The morphisms
$(u,S,v) : A \to A'$ are spans in $\CE$; that is, diagrams of the form $A\xleftarrow{u}S\xrightarrow{v} A'$.
The 2-cells $f : (u,S,v)\Rightarrow (r,T,s) : A \to A'$ are morphisms $S\xrightarrow{f} T$ such that
$r f = u$ and $s f = v$. 
Span composition is defined using pullback: 
$$(u',S',v')\ox (u,S,v) = (A \xra{(u,S,v)} A'\xra{(u',S',v')} A'') : = (A \xra{(up,P,v'q)} A'')$$ 
where the square in the diagram \eqref{spancomp} is a pullback. The identity morphism of $A$
is the span $A\xra{(1_A,A,1_A)}A$.
\begin{equation}\label{spancomp}
\begin{aligned}
\xymatrix{
& P \ar[r]^-{q} \ar[d]_-{p} & S' \ar[d]^-{u'} \ar[r]^-{v'} &A'' \\
A & S \ar[r]_-{v} \ar[l]^-{u} & A'}
\end{aligned}
\end{equation}

For any morphism $h : A\to B$ in $\CE$, we write $h_*$ and $h^*$ for
the morphisms $A\xrightarrow{(1_A,A,h)}B$ and $B\xrightarrow{(h,A,1_A)}A$, respectively,
in $\mathrm{Spn}(\CE)$ where we have an adjunction $h_*\dashv h^*$. 

\begin{remark}\label{CEgood} If $\CE$ is cocomplete and each morphism is exponentiable (= powerful) then
the bicategory $\mathrm{Spn}(\CE)$ admits colimits in the hom categories, all right extensions,
and all right liftings. So $\mathrm{Spn}(\CE)$ satisfies the local cocompleteness hypothesis
on the $\CV$ in Theorem~\ref{thm:comonadicModU}, for example.    
\end{remark}

Each functor $U : \CF \to \CE$ between categories with pullbacks determines an oplax functor
$\CU : \mathrm{Spn}(\CF) \to \mathrm{Spn}(\CE)$ agreeing with $U$ on objects and defined
on hom categories by the functors
$\CU^{X'}_X : \mathrm{Spn}(\CF)(X,X') \to \mathrm{Spn}(\CE)(UX,UX')$
which take the span 
$(r,T,s) :  X\to X'$ in $\CF$ to the span $(Ur,UT,Us) :  UX\to UX'$ in $\CE$.

The following is routine.
\begin{lemma}\label{Local_adj_span} 
Suppose $U : \CF \to \CE$ is a functor between categories with pullbacks.
If the functor $U$ has a right adjoint
$R$ with unit $\rho : 1_{\CF}\Rightarrow R U$ then the functor $\CU^{X'}_X$ has a right adjoint $\CR^{X'}_X$
defined by $\CR^{X'}_X(u,S,v) = (\rho_{X'})^*\otimes (Ru,RS,Rv)\otimes (\rho_{X})_*$.  
\end{lemma}
Our span analogue of Theorem~\ref{thm:adjointModU} follows from Proposition 2.7 of \cite{Kelly2002}
and our Lemma~\ref{Local_adj_span}.
\begin{proposition}\label{spanalog} 
Suppose $U : \CF \to \CE$ is a pullback preserving functor between categories with pullbacks.
Then $\CU : \mathrm{Spn}(\CF) \to \mathrm{Spn}(\CE)$ is a pseudofunctor.
If the functor $U$ has a right adjoint $R$ then $\CU$ has a right adjoint $\CR$ in $\mathrm{Caten}$
given on hom categories by the functors $\CR^{X'}_X$ of Lemma~\ref{Local_adj_span}.  
\end{proposition}

Let $G$ be a pullback-preserving comonad on a category $\CE$ which admits pullbacks. 
We shall define a two-sided enriched category $\CG : \mathrm{Spn}(\CE)\to \mathrm{Spn}(\CE)$.
The objects are the $G$-coalgebras $(A,A\xrightarrow{\gamma} GA)$ in $\CE$ with
$(A,\gamma)_-= (A,\gamma)_+= (A,\gamma)_0 : = A$. 
The functor $\CG_{(A,\gamma)}^{(A',\gamma')} : \mathrm{Spn}(\CE)(A,A')\to \mathrm{Spn}(\CE)(A,A')$
takes the span $(u,S,v) : A \to A'$ to the composite 
$$\gamma'^*\ox (Gu,GS,Gv)\ox \gamma_* : A \to A' \ . $$ 
Note that coactions are monomorphisms so the identity span $1_A$ of $A$ is a pullback of the cospan 
$A\xra{\gamma}GA\xla{\gamma}A$; so there is a unique invertible 2-cell $\eta_{(A,\gamma)} : 1_A\Rightarrow \CG_{(A,\gamma)}^{(A,\gamma)} 1_A$ in $\mathrm{Spn}(\CE)$.
The natural transformation $\mu^{(A',\gamma')}_{(A,\gamma),(A'',\gamma'')}$ has components defined by the composites
\begin{eqnarray*}
\lefteqn{\CG_{(A',\gamma')}^{(A'',\gamma'')}(u',S',v')\ox \CG_{(A,\gamma)}^{(A',\gamma')}(u,S,v)} \\
& = & \gamma''^*\ox (Gu',GS',Gv')\ox\gamma'_*\ox \gamma'^*\ox (Gu,GS,Gv)\ox \gamma_* \\
&\Rightarrow & \gamma''^*\ox (Gu',GS',Gv')\ox (Gu,GS,Gv)\ox \gamma_* \\
&\cong & \gamma''^*\ox (G(up),GP,G(v'q))\ox \gamma_* \\
& = & \CG_{(A,\gamma)}^{(A'',\gamma'')}((u',S',v')\ox (u,S,v))
\end{eqnarray*}
where the 2-cell is a whiskered counit of the adjunction $\gamma'_*\dashv \gamma'^*$
and the isomorphism comes from the fact that $G$ preserves the pullback \eqref{spancomp}.

Indeed, $\CG$ becomes a comonad on $\mathrm{Spn}(\CE)$ as follows. The natural transformations
$$\mathrm{d}_{(A,\gamma)}^{(A',\gamma')} : \CG_{(A,\gamma)}^{(A',\gamma')}\Lra (\CG_{(A,\gamma)}^{(A',\gamma')})^2 \ \text{  and  } \ \mathrm{e}_{(A,\gamma)}^{(A',\gamma')} : \CG_{(A,\gamma)}^{(A',\gamma')}\Lra 1_{\mathrm{Spn}(\CE)_{A}^{A'}}$$ 
have component at $(u,S,v)$ induced on the limits of the rows of the diagrams \eqref{induce_d} and
\eqref{induce_e} by those commuting diagrams. The limits exist in $\CE$ since they can be constructed by iterated pullbacks.
\begin{eqnarray}\label{induce_d}
\begin{aligned}
\xymatrix{
A \ar[d]_-{1_A} \ar[r]^-{\gamma}& GA  \ar[d]_-{\delta_A} & GS \ar[l]_-{Gu} \ar[d]^-{\delta_S} \ar[r]^-{Gv} & GA'  \ar[d]^-{\delta_{A'}} & A' \ar[l]_-{\gamma'} \ar[d]^-{1_{A'}}\\
A \ar[r]^-{G\gamma \circ \gamma} & G^2A &  G^2S \ar[l]_-{G^2u} \ar[r]^-{Gv} & G^2 A' & A' \ar[l]_-{G\gamma'}}
\end{aligned}
\end{eqnarray}
\begin{eqnarray}\label{induce_e}
\begin{aligned}
\xymatrix{
A \ar[d]_-{1_A} \ar[r]^-{\gamma}& GA  \ar[d]_-{\varepsilon_A} & GS \ar[l]_-{Gu} \ar[d]^-{\varepsilon_S} \ar[r]^-{Gv} & GA'  \ar[d]^-{\varepsilon_{A'}} & A' \ar[l]_-{\gamma'} \ar[d]^-{1_{A'}}\\
A \ar[r]^-{1_A} & A &  S \ar[l]_-{u} \ar[r]^-{v} &  A' & A' \ar[l]_-{1_A}}
\end{aligned}
\end{eqnarray}

\begin{proposition}\label{pbcomonad}
For any category $\CE$ with pullbacks and any comonad $G$ on $\CE$ which preserves pullbacks, the comonad $\CG$ just constructed satisfies an equivalence $$\mathrm{Spn}(\CE)^{\CG} \simeq \mathrm{Spn}(\CE^{G})$$ in $\mathrm{Caten}$.
\end{proposition}
\begin{proof}
Note that the category $\CE^{G}$ of $G$-coalgebras has pullbacks since $\CE$ does
and $G$ preserves them. 
The displayed bicategories do have the same objects, namely, the $G$-coalgebras.
By definition, $\mathrm{Spn}(\CE)^{\CG}((A,\gamma),(A',\gamma')) = \mathrm{Spn}(\CE)(A,A')^{\CG_{(A,\gamma)}^{(A',\gamma')}}$. An object of this category is a $\CG_{(A,\gamma)}^{(A',\gamma')}$-coalgebra; that is, a span $(u,S,v) : A \to A'$ together with a span morphism
$(u,S,v)\to \gamma'^*\ox (Gu,GS,Gv)\ox \gamma_*$ satisfying two axioms for a coaction.
Such span morphisms are in bijection with morphisms $\sigma : S\to GS$ in $\CE$ such that the
diagram
\begin{eqnarray}\label{sigmaG-coalg}
\begin{aligned}
\xymatrix{
A  \ar[d]_-{\gamma} && S \ar[ll]_-{u} \ar[rr]^-{v} \ar@{.>}[d]^-{\sigma} && A' \ar[d]^-{\gamma'} \\
GA  && GS \ar[ll]^-{Gu} \ar[rr]_-{Gv}  && GA'
} 
\end{aligned}
\end{eqnarray}
commutes while the two coaction conditions translate to the two conditions for $(S,\sigma)$ to be
a $G$-coalgebra. So \eqref{sigmaG-coalg} tells us that $u$ and $v$ are $G$-coalgebra morphisms and we have a span in $\CE^G$. This extends to an isomorphism of categories $$\mathrm{Spn}(\CE)(A,A')^{\CG_{(A,\gamma)}^{(A',\gamma')}} \cong \mathrm{Spn}(\CE^G)((A,\gamma),(A',\gamma')) \ .$$
It is easy to see that composition and identities in the bicategories correspond under these isomorphisms.
\end{proof}
\begin{corollary} If the functor $U : \CF\to \CE$ of Proposition~\ref{spanalog} is comonadic in $\mathrm{Cat}$ then
$\CU : \mathrm{Spn}(\CF)\to \mathrm{Spn}(\CE)$ is comonadic in  $\mathrm{Caten}$.
\end{corollary}
\begin{example}\label{interioroperator}
Take any set $X$ and let $\CE = \mathcal{P}X$ be the ordered set of subsets of $X$.
This $\CE$ is complete; pullbacks are intersections. 
We have the locally ordered bicategory $\CS X = \mathrm{Spn}(\CE)$.     
Now let $X$ be a topological space. Write $\mathcal{O}X$ for the ordered set of open subsets of $X$. 
For our example, take $G$ to be the interior operator for the topology. 
This $G$ is a finite-limit-preserving idempotent comonad on the complete cartesian-monoidal category $\CE = \mathcal{P}X$. 
We have $\CE^{\CG} = \mathcal{O}X$.
By Proposition~\ref{pbcomonad}, the bicategory $\mathrm{Spn}(\mathcal{O}X)$ is 
comonadic over $\CS X$ in $\mathrm{Caten}$.  
\end{example}

\section{Differential comonads}\label{DiffComonads}

Let $\CV$ be a locally additive bicategory; that is, a bicategory whose homs are additive categories such that composition with a morphism, on either side, is an additive functor.

\begin{definition}
A {\em differential system $\CD$ on $\CV$} consists of a set $\mathrm{ob}\CD$ (whose elements are called {\em objects of $\CD$}), a function $(-)_0 :  \mathrm{ob}\CD\to \mathrm{ob}\CV$, 
additive functors $\CD^{D'}_D : \CV^{D'_0}_{D_0} \to \CV^{D'_0}_{D_0}$ for all objects $D, D'$ of $\CD$, and families of 2-cells 
$$\mathrm{r}_{v',v} : \CD^{D''}_D(v')\otimes v\to \CD^{D''}_D(v'\otimes v) \ \text{  and  } \ 
\mathrm{\ell}_{v',v} : v'\otimes \CD^{D'}_D(v)\to \CD^{D''}_D(v'\otimes v) \ ,$$
natural in $v\in \CV_D^{D'}, v'\in \CV_{D'}^{D''}$, satisfying conditions \eqref{ax_r}-\eqref{eq:rlcomp} (in which
brackets, associativity constraints, and unit constraints are omitted for easier reading). 
\begin{eqnarray}\label{ax_r}
\begin{aligned}
\xymatrix{
\CD^{D'''}_D(v'')\ox v'\ox v \ar[rd]_{\mathrm{r}_{v'',v'\ox v}}\ar[rr]^{\mathrm{r}_{v'',v'}\ox 1_v}   && \CD^{D'''}_D(v''\ox v')\ox v \ar[ld]^{\ \mathrm{r}_{v''\ox v',v}} \\
& \CD^{D'''}_D(v''\ox v'\ox v)  &
}
\end{aligned}
\end{eqnarray}
\begin{eqnarray}\label{ax_l}
\begin{aligned}
\xymatrix{
v''\ox v'\ox \CD^{D'}_D(v) \ar[rd]_{\mathrm{\ell}_{v''\ox v',v}}\ar[rr]^{1_{v''}\ox\mathrm{\ell}_{v',v}}   && v''\ox \CD^{D''}_D(v'\ox v) \ar[ld]^{\mathrm{\ell}_{v'',v'\ox v}} \\
& \CD^{D'''}_D(v''\ox v'\ox v)  &
}
\end{aligned}
\end{eqnarray}
\begin{eqnarray}\label{ax_rl}
\begin{aligned}
\xymatrix{
v''\ox \CD^{D''}_{D'}(v')\ox v \ar[rr]^-{\mathrm{\ell}_{v'',v'}\ox 1_v} \ar[d]_-{1_{v''}\ox \mathrm{r}_{v',v}} &&  \CD^{D'''}_{D'}(v'')\ox v'\ox v \ar[d]^-{\mathrm{r}_{v'',v'\ox v}} \\
v''\ox v' \ox \CD^{D'}_{D}(v) \ar[rr]_-{\mathrm{\ell}_{v''\ox v',v}} && \CD^{D'''}_D(v''\ox v'\ox v)}
\end{aligned}
\end{eqnarray}
\begin{eqnarray}\label{ax_unit}
\begin{aligned}
\CD^{D'}_{D}(v)\xra{1_{\CD^{D'}_{D}(v)}} \CD^{D'}_{D}(v) \
& = &   \CD^{D'}_{D}(v)\ox 1_{D_0}\xra{\mathrm{r}_{v,1_{D_0}}} \CD^{D'}_{D}(v\ox 1_{D_0}) \\ 
& = & 1_{D'_0}\ox\CD^{D'}_{D}(v)\xra{\mathrm{\ell}_{1_{D_0},v}} \CD^{D'}_{D}(1_{D'_0}\ox v)
\end{aligned}
\end{eqnarray}
\begin{equation}\label{eq:rlcomp}
\CD^{D''}_{D}(\mathrm{\ell}_{v',v})\circ  r_{v',\CD^{D'}_{D}(v)} + \CD^{D''}_{D}(\mathrm{r}_{v',v})\circ l_{\CD^{D''}_{D'}(v'),v}=0
\end{equation}
\end{definition}

\begin{proposition}\label{diff_comonad}
Suppose $\CD$ is a differential system on the locally additive bicategory $\CV$ whose
hom categories have finite direct sums. The following defines a comonad $\CG$ on $\CV$ in
$\mathrm{Caten}$. The object span of $\CG$ is that of $\CD$. The functor $\mathscr{G}_D^{D'}:\mathscr{V}_{D_0}^{D'_0}\rightarrow \mathscr{V}_{D_0}^{D'_0}$ is defined by 
$$\mathscr{G}_D^{D'} = 1_{\CV_{D_0}^{D'_0}}\oplus \CD_D^{D'} \ .$$
The 2-cell $\mu_{v',v}$ is defined by the matrix 
$\begin{bmatrix}
1 & 0 & 0 & 0 \\
0 & \mathrm{\ell}_{v',v} & \mathrm{r}_{v',v} & 0
\end{bmatrix}$ $:$ 
 \begin{eqnarray*}
 (v'\ox v)\oplus (v'\ox \CD_D^{D'}(v))\oplus (\CD_{D'}^{D''}(v')\ox v)\oplus (\CD_{D'}^{D''}(v')\ox \CD_{D}^{D'}(v)) \Rightarrow (v'\ox v)\oplus  \CD_D^{D''}(v'\ox v) \ .
\end{eqnarray*}
The 2-cell $\eta :1_{D_0}\Rightarrow 1_{D_0}\oplus \mathscr{D}_D^D(1_{D_0})$ is the first injection $\begin{bmatrix}
1 \\
0 
\end{bmatrix}$.
The 2-cell $\mathrm{d}_v$ is the matrix
$$\begin{bmatrix}
1 & 0  \\
0 & 1 \\
0 & 1 \\
0 & 0
\end{bmatrix}
 : v\oplus \CD_D^{D'}(v) \Rightarrow v\oplus \CD_D^{D'}(v) \oplus \CD_D^{D'}(v)\oplus (\CD_D^{D'})^2(v) \ .$$ 
 The 2-cell $\mathrm{e}_v$ is the first projection
$$\begin{bmatrix}
1 & 0  
\end{bmatrix}
 : v\oplus \CD_D^{D'}(v) \Rightarrow v \ .$$   
\end{proposition}
\begin{proof}
Associativity of composition $\mu$ for $\CG$ is proved by calculating two matrix products:
\begin{eqnarray*}
\begin{bmatrix}
1 & 0 & 0 & 0 \\
0 & \mathrm{\ell} & \mathrm{r} & 0
\end{bmatrix}
\begin{bmatrix}
1 & 0 & 0 & 0 & 0 & 0 & 0 & 0 \\
0 & 1\ox \mathrm{\ell} & 1\ox \mathrm{r} & 0 & 0 & 0 & 0 & 0 \\
0 & 0 & 0 & 1 & 0 & 0 & 0 & 0 \\
0 & 0 & 0 & 0 & 0 & 1\ox \mathrm{\ell} & 1\ox \mathrm{r} & 0
\end{bmatrix}
\end{eqnarray*}
\begin{eqnarray*}
\begin{bmatrix}
1 & 0 & 0 & 0 \\
0 & \mathrm{\ell} & \mathrm{r} & 0
\end{bmatrix}
\begin{bmatrix}
1 & 0 & 0 & 0 & 0 & 0 & 0 & 0 \\
0 & 1 & 0 & 0 & 0 & 0 & 0 & 0 \\
0 & 0 & \mathrm{\ell}\ox 1 & \mathrm{r}\ox 1 & 0 & 0 & 0 & 0 \\
0 & 0 & 0 & 0 &  \mathrm{\ell}\ox 1 & \mathrm{r}\ox 1 & 0 & 0
\end{bmatrix}
\end{eqnarray*}
which are equal by \eqref{ax_r}, \eqref{ax_l} and \eqref{ax_rl}.
That $\eta$ is the unit for $\mu$ is proved by calculating two matrix products of the form:
\begin{eqnarray*}
\begin{bmatrix}
1 & 0 & 0 & 0 \\
0 & \mathrm{\ell} & \mathrm{r} & 0
\end{bmatrix}
\begin{bmatrix}
1 & 0  \\
0 & 0 \\
0 & 1 \\
0 & 0 
\end{bmatrix}
\ \text{   ,   } \
\begin{bmatrix}
1 & 0 & 0 & 0 \\
0 & \mathrm{\ell} & \mathrm{r} & 0
\end{bmatrix}
\begin{bmatrix}
1 & 0  \\
0 & 1 \\
0 & 0 \\
0 & 0 
\end{bmatrix}
\end{eqnarray*}
which are both identity $2\times 2$-matrices by \eqref{ax_unit}.

The proof that $\mathrm{d}$ and $\mathrm{e}$ are enriched functors and that they satisfy
coassociativity and counital conditions for a comonad proceeds similarly using matrix
multiplication. Only the condition relating $\mathrm{d}$ and $\mu$ (see \cite{Kelly2002} diagram (2.16))
requires any of the differential system axioms, namely, condition \eqref{eq:rlcomp};
for a little more detail on this, see \cite{NikolicPhD}.
\end{proof}

A comonad arising as in Proposition~\ref{diff_comonad} is called a {\em differential comonad}.

\begin{example}\label{E}
Let $\CC$ be a braided monoidal additive category containing an object $E$ with
$$\sigma_{E,E} = -1_{E\ox E} : E\ox E \to E\ox E \ ,$$
where $\sigma_{X,Y} : X\ox Y\to Y\ox X$ is the braiding for $\CC$. 
Take $\CV$ to be the one-object bicategory with endohom $\CC$.
The following defines a differential system $\CD$ on $\CV$ (where we are writing as if $\CC$ were strict monoidal):
\begin{itemize}
\item[a.] $\mathrm{ob}\CD$ is a singleton $\{ D \}$;
\item[b.] $\CD =\CD_D^D : = E\ox - : \CC \to \CC$;
\item[c.] $\mathrm{r}_{Y,X} = 1_{E\ox Y\ox X} : E\ox Y\ox X \to E\ox Y\ox X$;
\item[d.] $\mathrm{r}_{Y,X} = \sigma_{Y,E}\ox 1_X : Y\ox E\ox X \to E\ox Y\ox X$.
\end{itemize}
All the differential system conditions are obvious except that \eqref{eq:rlcomp} translates to
the equation $\sigma_{E\ox Y,E} + \sigma_{Y,E} = 0$ which holds by the property
$\sigma_{E\ox Y,E} = (\sigma_{E,E}\ox 1_Y)\circ (1_E\ox \sigma_{Y,E})$ of the braiding and the condition $\sigma_{E,E} = -1_{E\ox E}$. 
\end{example} 

We have the following easy identification of Eilenberg-Moore coalgebra construction for a differential comonad in terms of the differential system.

\begin{proposition}\label{coalgs_diff_comonad}
The bicategory $\CV^{\CG}$ for the differential comonad $\CG$ on $\CV$ as in Proposition~\ref{diff_comonad} has the same objects as the differential system $\CD$. 
The objects of the hom category $\CV^{\CG}(D,D')$ are pairs $(x,\delta)$ where $x : D_0\to D'_0$
is a morphism and $\delta : x \Rightarrow \CD_D^{D'}(x)$ is a 2-cell of $\CV$ satisfying 
$\CD_D^{D'}(\delta)\delta = 0$. Morphisms $\theta : (x,\delta)\Rightarrow (x_1,\delta_1)$ are 2-cells 
$\theta : x \Rightarrow x_1$ in $\CV$ such that $\theta \delta = \delta_1 \theta$. 
The composite $D\xra{(x,\delta)} D'\xra{(x',\delta')} D''$ is $D\xra{(x'\ox x,\mathrm{r}_{x',x}(\delta' \ox 1_x)+\mathrm{\ell}_{x',x}(1_{x'}\ox\delta))} D'' $.
\end{proposition}

\begin{example}\label{Eip}
Consider the case of Example~\ref{E} where $\CC = \mathrm{GAb}$ is the symmetric monoidal
category of $\mathbb{Z}$-graded abelian groups. The objects are families $A = (A_n)_{n\in \mathbb{Z}}$ of abelian groups with the morphisms families of abelian group morphisms.
Elements of $A_n$ are said to be of {\em degree} $n$.
We regard abelian groups as objects by putting them in degree $0$ and putting the abelian group $0$
in all other degrees.
The tensor product (as per \cite{EilKel1966}) is defined by the convolution formula 
$$(A\ox B)_n = \sum_{p+q=n}{A_p\ox B_q} \ .$$
The unit object for this tensor product is the group $\mathbb{Z}$ of integers under addition.
The symmetry $\sigma = \sigma_{A,B} : A\ox B \to B\ox A$ is defined by
$$\sigma(a\ox b) = (-1)^{pq}b\ox a$$
for $a\in A_p$ and $b\in B_q$. The {\em suspension} of $A\in \mathrm{GAb}$ is the 
graded abelian group $\mathrm{S}A$ defined by $\mathrm{S}A_n= A_{n-1}$. 
Then $E =  \mathrm{S}\mathbb{Z}$ has the property that $\sigma_{E,E} = -1_{E\ox E}$. 
Proposition~\ref{diff_comonad} provides a monoidal comonad on $\mathrm{GAb}$
arising from this $E$ (in fact, since $E\ox A \cong \mathrm{S}\mathbb{A}$, we have $\CD \cong \mathrm{S}$
and Proposition~\ref{coalgs_diff_comonad} tells us that the monoidal category of Eilenberg-Moore
coalgebras is the category $\mathrm{DGAb}$ of chain complexes of abelian groups with the 
usual symmetric monoidal structure (see \cite{EilKel1966}). This example can be generalised in the
obvious way to replace $\mathrm{Ab}$ by any monoidal additive category with countable coproducts.      
\end{example} 

\section{Comonadicity of comonadic base change}\label{Ccbc}
Recall the base change result in \cite{Kelly2002} which appears there as Proposition 7.5 and states that right whiskering modules by a two-sided enriched category provides a normal lax functor.
\begin{theorem}\label{thm:comonadicModU}
Let $\CG$ be a comonad in $\mathrm{Caten}$ on the locally cocomplete bicategory $\mathscr{V}$ and let ${\CU}\dashv {\CR}$ be as in diagram \eqref{diag:catenAdj}. Then the lax functor 
\begin{align*}
\mathscr{N}:= \mathrm{Moden}(\mathscr{X},\mathscr{V}^\mathscr{G})
 \xrightarrow{\widetilde{\mathscr{U}}:=\mathrm{Moden}(\mathscr{X},\mathscr{U})}
 \mathscr{M:=}\mathrm{Moden}(\mathscr{X},\mathscr{V}) \ ,
\end{align*}
given by right whiskering with ${\CU}$, is comonadic in\footnote{Note that $\mathscr{N}$ and $\mathscr{M}$ may have a large set (proper class) of objects.} $\mathrm{CATEN}$. 
\end{theorem}
\begin{proof}
As in the proof of Theorem~\ref{thm:adjointModU}, we only look at $\mathscr{N}=\mathscr{V}^\mathscr{G}\text{-Mod}$, and $\mathscr{M}=\mathscr{V}\text{-Mod}$.
By Theorem~\ref{thm:adjointModU}, we have a right adjoint $\widetilde{\mathscr{R}}$ to $\widetilde{\mathscr{U}}$. Moreover, $\widetilde{\mathscr{U}}$ satisfies the other two conditions of Proposition \ref{prop:comonadic} as stated in Lemma \ref{lem:whisk}. This is all we need to conclude that $\widetilde{\mathscr{U}}$ is comonadic.
\end{proof}
\begin{remark}\label{monoidcasecomonadic}
Theorem~\ref{thm:comonadicModU} restricts to the one-object case.
In the situation of Remark~\ref{monoidcaseadjoint}, if $\CU : \CD\to \CC$ is comonadic in $\mathrm{Cat}$
then $\widetilde{\CU}$ is comonadic in $\mathrm{Caten}$.
\end{remark}

\begin{example}\label{Spnenrich}
Recall from Example 2.1 of \cite{88} that each Grothendieck fibration $P : \CF \to \CE$ 
(in the sense of \cite{GroSGA}), over the category $\CE$ with pullbacks,
gives rise to a $\mathrm{Spn}(\CE)$-enriched category $\CA$ when the fibration is locally small (or ``has small homs'') in the two-sided sense; see \cite{Ben1975, 32}. 
The objects of $\CA$ over $X$ are the objects $A$ of $\CF$ with $PA=X$.
The span $\CA(A,B) : X\to Y$, where $PA=X$, $PB=Y$, is defined universally (using local smallness) by a natural bijection
$$\mathrm{Spn}(\CE)(X,Y)((u,S,v),\CA(A,B)) \cong \CF_S(u^{\star}A,v^{\star}B)$$ where $\CF_S$ is the fibre
of $P$ over $S$ and $u^{\star}A\to A$, $v^{\star}B\to B$ are cartesian morphisms over $u : S\to X$, $v : S \to Y$,
respectively. It follows from Theorem~\ref{thm:comonadicModU} that the $\mathrm{Spn}(\CE^G)$-enriched category arising from a locally small fibration over $\CE^G$ is a coalgebra for a base-change comonad on $\mathrm{Spn}(\CE)\text{-}\mathrm{Mod}$. 
\end{example}

\begin{example} Let $\Delta$ denote the topologist's simplex category: objects are finite non-empty linearly ordered sets
$[n] = \{0,1,\dots,n\}$, for $n\ge 0$, and morphisms are order-preserving functions. Let $\mathbb{N}$ be the discrete
category of natural numbers. Restriction along the bijective-on-objects functor $[-] : \mathbb{N}\to \Delta^{\mathrm{op}}$
taking $n$ to $[n]$ is a comonadic functor $U : [\Delta^{\mathrm{op}}, \mathrm{Set}] \to [\mathbb{N}, \mathrm{Set}]$ since it is conservative and has adjoints on both sides given by left and right Kan extension.
Using the formula for right Kan extension, we see that the right adjoint $R$ for $U$ is given by
\begin{eqnarray*}
R(X)[n] = \prod_{\xi\in \Delta([u],[n])} X_u \ \text{  and  } \ R(X)([m]\xra{\theta}[n])(x_{\xi})_{\xi}= (x_{\theta\zeta})_{\zeta} \ .
\end{eqnarray*}
Since $U$ preserves products, it is strong monoidal for the cartesian monoidal structures on its domain and codomain categories. 
Of course, $[\Delta^{\mathrm{op}}, \mathrm{Set}]$ is the category of simplicial sets; categories enriched therein are of interest in algebraic topology and higher category theory (for example, see \cite{CorPor}).  
We deduce from Theorem~\ref{thm:comonadicModU} that the bicategory of simplicial-set-enriched categories and modules between them is comonadic in $\mathrm{CATEN}$ over the bicategory of $\mathbb{N}$-graded-set-enriched categories and modules between them.  
\end{example}

\begin{example} Let $\CV$ be a locally presentable symmetric monoidal category and let $\mathrm{Comon}\CV$ denote the
category of comonoids and comonoid morphisms in $\CV$. 
Then $\mathrm{Comon}\CV$ is locally-presentable symmetric monoidal
with the forgetful functor $U : \mathrm{Comon}\CV \to \CV$ comonadic and symmetric strong monoidal; see \cite{Porst2008}.   
We deduce from Theorem~\ref{thm:comonadicModU} that the bicategory of $\mathrm{Comon}\CV$-enriched categories and modules between them is comonadic in $\mathrm{CATEN}$ over the bicategory $\CV\text{-}\mathrm{Mod}$ of $\CV$-enriched categories and modules between them.  
As shown in \cite{Vasil2012}, an example of a $\mathrm{Comon}\CV$-enriched category  
is the category $\mathrm{Mon}\CV$ of monoids and monoid morphisms in $\CV$; 
the homs are given by the measuring comonoid construction; see \cite{AneJoy2013, HyLFVa2017} for consequential applications and developments. 
\end{example}

\begin{example}\label{threecomonads} 
Consider the diagram~\eqref{diag:adjs} of monoidal additive categories and additive functors.
\begin{equation}\label{diag:adjs}
\begin{tikzpicture}[scale=2, every node/.style={scale=1},baseline=(current  bounding  box.center)]
\node (d) at (-2,0) {$\mathrm{DGAb}$};
\node (g) at (0,0) {$\mathrm{GAb}$};
\node (a) at (2,0) {$\mathrm{Ab}$};
\node (sym1) at (-1,0.2) {$\bot$};
\node (sym2) at (-1,-0.2) {$\bot$};
\node (sym3) at (1,-0.2) {$\bot$};
\path[,->,font=\scriptsize,>=angle 90,bend right]
(g) edge node[above] {$L$} (d);
\path[,->,font=\scriptsize,>=angle 90]
(d) edge node[below left] {$U$} (g);
\path[->,font=\scriptsize,>=angle 90,bend left]
(g)edge node[below] {$R$} (d);
\path[,->,font=\scriptsize,>=angle 90,bend left]
(a) edge node[below] {$C$} (g);
\path[,->,font=\scriptsize,>=angle 90]
(g) edge node[above] {$\Sigma$} (a);
%
\end{tikzpicture}
\end{equation}
Some of the notation was explained in Example~\ref{Eip}. Of course, $\mathrm{Ab}$ is the monoidal
category of abelian groups so that $\mathrm{Ab}$-enriched categories are additive categories (no direct
sums required). Categories enriched in $\mathrm{GAb}$ are graded categories or G-categories while
categories enriched in $\mathrm{DGAb}$ are differential graded categories or DG-categories. 
These were motivating examples of enriched categories for Eilenberg-Kelly \cite{EilKel1966}.
The functor $U$ forgets the differentials in the chain complexes and has adjoints $L\dashv U\dashv R$
given by
\begin{align}
L(C)_n =C_{n+1}\oplus C_n \ , \
R(C)_n =C_{n}\oplus C_{n-1} \ , \
d =\left[\begin{array}{cc}
    0 & 1	\\
    0 & 0 
    \end{array}\right]\,.
\end{align} 
The functor $\Sigma$ takes $A=(A_n)_{n\in \mathbb{Z}}$ to the coproduct $\Sigma A = \sum_{n\in \mathbb{Z}}A_n$ while its right adjoint $C$ takes each abelian group $X$ to the graded abelian group $CX$ with $X$
in all degrees. We already mentioned in Example~\ref{Eip} that $U$ is comonadic but so too are $\Sigma$
and $\Sigma\circ U$. The comonad $\Sigma \circ C$ generated by $\Sigma \dashv C$ is given by tensoring
with the Hopf ring $\mathbb{Z}[x,x^{-1}]$ of Laurent polynomials with integer coefficients. 
The comonad $\Sigma \circ U \circ R\circ C$ generated by $\Sigma \circ U \dashv R\circ C$ is given by tensoring with the Hopf ring which is the object of study by Pareigis in \cite{Pareigis1981}. 
So we obtain three applications of Theorem~\ref{thm:comonadicModU}. To be explicit, the
pseudofunctors 
\begin{eqnarray*}\label{3comonadics}
\mathrm{GAb}\text{-}\mathrm{Mod} \xra{\widetilde{\Sigma}}\mathrm{Ab}\text{-}\mathrm{Mod} \ , \
\mathrm{DGAb}\text{-}\mathrm{Mod} \xra{\widetilde{U}}\mathrm{GAb}\text{-}\mathrm{Mod} \ , \ 
\mathrm{DGAb}\text{-}\mathrm{Mod} \xra{\widetilde{U\circ \Sigma}}\mathrm{Ab}\text{-}\mathrm{Mod}
\end{eqnarray*}
are all comonadic in $\mathrm{CATEN}$.
\end{example}
\begin{example}\label{interioroperatorbasechange}
Refer back to the situation of Example~\ref{interioroperator}.
For $X$ just a set, we will give an interpretation of the bicategory $\CS X\text{-}\mathrm{Mod}$
in terms of families of ordered sets over the discretely ordered $X$. Let $\mathrm{Idl}/X$ denote the 2-category of
ordered objects and two-sided order ideals (in the sense of \cite{27}) in the topos $\mathrm{Set}/X$ of sets over $X$.
The objects are ordered sets $A$ equipped with a function $p : A\to X$ such that $a\le a'$ implies $p(a)=p(a')$.
The morphisms $M : (A,p) \to (B,q)$ are relations $M\subseteq B\times A$ such that 
\begin{itemize}
\item[(i)] $bMa$ implies $p(a)=q(b)$ and
\item[(ii)] $bMa, b'\le b, a\le a'$ imply $b'Ma'$
\end{itemize}
where $bMa$ means $(b,a)\in M$.  
The 2-cells are inclusions. Morphisms are composed as relations. 

An invertible 2-functor
 \begin{eqnarray*}
\Gamma : \mathrm{Idl}/X \lra \CS X\text{-}\mathrm{Mod} 
\end{eqnarray*}
is defined as follows. The $\CS X$-category $\Gamma (A,p)$ has objects $(U,s)$ where $s$ is a section of
$p$ over $U\subseteq X$; that is, $s : U\to A$ is a function such that $p(s(x))=x$ for all $x\in U$.
Put $\Gamma (A,p)((U,s),(U',s')) = \{x \in X : s(x)=s'(x) \}\subseteq U\cap U'$ as a span from $U$ to $U'$ in $\mathcal{P}X$. The $\CS X$-category composition and identity inclusions are obvious. For an ideal $M : (A,p) \to (B,q)$, the
module $\Gamma M : \Gamma (A,p) \to \Gamma (B,q)$ is defined by 
$\Gamma M ((V,t),(U,s)) = \{x \in X : t(x)Ms(x) \}$. The inverse for $\Gamma$ takes the $\CS X$-category $\mathbf{A}$
to the set of objects $\mathbf{a}$ over singleton subsets $\mathbf{a}_+$ of $X$ ordered by $\mathbf{a}\le \mathbf{a}'$ if and only
if $\mathbf{a}$ and $\mathbf{a}'$ are over the same $\{ x\}$ and $x\in \mathbf{A}(\mathbf{a},\mathbf{a}')$.      

Again let $X$ be a topological space with $G$ to be the interior operator for the topology. 
By Theorem~\ref{thm:comonadicModU}, the bicategory $\mathrm{Spn}(\mathcal{O}X)\text{-}\mathrm{Mod}$ is 
comonadic over $\mathrm{Idl}/X$ in $\mathrm{Caten}$.  
Recall that Walters \cite{Walters1981} showed that sheaves on the space $X$ are equivalent to symmetric Cauchy
complete categories enriched in this bicategory $\mathrm{Spn}(\mathcal{O}X)$.
Hence, the bicategory of left adjoint morphisms in $\mathrm{Spn}(\mathcal{O}X)\text{-}\mathrm{Mod}$ is of interest. 
\end{example}

\section{Fusion operators and Hopfness for comonads in $\mathrm{Caten}$}\label{sec:Fusion}

Fusion operators in monoidal categories were studied in \cite{Street1998} along with the example of
those coming from bimonoids. A proof that a bimonoid has an antipode if and only if the fusion operator
is invertible can be found in \cite{Booker2011} Appendix.
In \cite{Bruguieres2011}, an opmonoidal monad on a monoidal category was called Hopf when a
suitable fusion morphism was invertible. 
The dual notion of Hopf monoidal comonad was studied in a general context in \cite{Chikhladze2010}.   
Here we identify Hopf comonads in $\text{Caten}$.

\begin{definition}
For a comonad $\mathscr{G}$ on $\CV$ in $\mathrm{Caten}$, the {\em left fusion operator} at $v'\in \mathscr{V}(G'_0,G''_0)$ and $(v,\gamma_v)\in \mathscr{V}^\mathscr{G}(G,G')$ is
the composite 2-cell
\begin{equation*}
 \mathrm{v}_{v',v}^\mathrm{l}:\mathscr{G}_{G'}^{G''}v'\otimes v
 \xrightarrow{1\otimes \gamma_v} 
 \mathscr{G}_{G'}^{G''} v'\otimes \mathscr{G}_{G}^{G'}v
 \xrightarrow{(\mu_{G,G''}^{G'})_{v',v}}
 \mathscr{G}_{G}^{G''}(v'\otimes v) \ .
\end{equation*}
Dually, in the sense of \cite{Kelly2002} Section 2.9, the {\em right fusion operator} at $v'\in \mathscr{V}(G'_0,G''_0)$ and $(v,\gamma_v)\in \mathscr{V}^\mathscr{G}(G,G')$ is
the composite 2-cell
\begin{equation*}
 \mathrm{v}_{v',v}^\mathrm{r}:
 v'\otimes \mathscr{G}_{G}^{G'}v
 \xrightarrow{\gamma_{v'}\otimes 1} 
 \mathscr{G}_{G'}^{G''} v'\otimes \mathscr{G}_{G}^{G'}v
 \xrightarrow{(\mu_{G,G''}^{G'})_{v',v}}
 \mathscr{G}_{G}^{G''}(v'\otimes v) \ .
\end{equation*}
The comonad $\CG$ is called {\em left (right) Hopf} when the left (right) fusion operators are all
invertible.
\end{definition}

Since we will use left fusion operators mostly, we simply write $\mathrm{v}_{v',v}$ for $\mathrm{v}_{v',v}^\mathrm{l}$.

\begin{proposition}
The inverse fusion maps are $\mathscr{G}$-compatible in the first variable
\begin{equation*}
\begin{tikzcd}
\mathscr{G}(v'\otimes v) 
		  \arrow[rightarrow]{dr}{\mathrm{v}_{v',v}^{-1}}& \\
v'\otimes v \arrow[leftarrow]{u}{\mathrm{e}_{v'\otimes v}}
	&	 \mathscr{G} v' \otimes v 	
		 	  \arrow[rightarrow]{l}{\mathrm{e}_{v'}\otimes 1}
\end{tikzcd}
\begin{tikzcd}
\mathscr{G}(v'\otimes v) 
		  \arrow[rightarrow]{r}{\mathrm{v}_{v',v}^{-1}}&
		 \mathscr{G} v' \otimes v 	
		 	  \arrow[rightarrow]{r}{\mathrm{d}_{v'}\otimes 1}
& \mathscr{G}^2 v' \otimes v 
 \arrow[leftarrow]{d}{\mathrm{v}_{\mathscr{G}v',v}^{-1}}
\\
\mathscr{G}^2(v'\otimes v)
\arrow[leftarrow]{u}{\mathrm{d}_{v'\otimes v}}
\arrow[rightarrow]{rr}{\mathscr{G}\mathrm{v}_{v',v}^{-1}}
& 
& \mathscr{G}(\mathscr{G} v' \otimes v )
\end{tikzcd}
\end{equation*}
as well as compatible with any coalgebra structure existing on $v'$, in the sense that
\begin{equation*}
\begin{tikzcd}
\mathscr{G}(v'\otimes v) 
	\arrow[rightarrow]{dr}{\mathrm{v}_{v',v}^{-1}}
&
\\
v'\otimes v
	\arrow[rightarrow]{u}{\gamma_{v'\otimes v}}
& \mathscr{G} v' \otimes v 	\;.
	\arrow[leftarrow]{l}{\gamma_{v'}\otimes 1}
\end{tikzcd}
\end{equation*}
\end{proposition}
\begin{proof}
Refer to the commuting diagrams \eqref{me1}, \eqref{me2}, \eqref{me3}. 
\end{proof}
\begin{equation}\label{me1}
\begin{tikzcd}
\mathscr{G}(v'\otimes v) 
		  \arrow[leftarrow]{r}{\mu_{v',v}}& 
\mathscr{G}v'\otimes \mathscr{G}v 
\arrow[leftarrow]{d}{1\otimes\gamma_v}
\arrow[rightarrow]{dl}{\mathrm{e}_{v'}\otimes\mathrm{e}_{v}}
\\
v'\otimes v \arrow[leftarrow]{u}{\mathrm{e}_{v'\otimes v}}
	&	 \mathscr{G} v' \otimes v 	
		 	  \arrow[rightarrow]{l}{\mathrm{e}_{v'}\otimes 1}
\end{tikzcd}
\end{equation}
\begin{equation}\label{me3}
\begin{tikzcd}
\mathscr{G}(v'\otimes v) 
	\arrow[leftarrow]{r}{\mu_{v',v}}
& \mathscr{G}v'\otimes \mathscr{G}v 
	\arrow[leftarrow]{d}{1\otimes\gamma_v}
	\arrow[leftarrow]{dl}{\gamma_{v'}\otimes\gamma_{v}}
\\
v'\otimes v
	\arrow[rightarrow]{u}{\gamma_{v'\otimes v}}
& \mathscr{G} v' \otimes v 	
	\arrow[leftarrow]{l}{\gamma_{v'}\otimes 1}
\end{tikzcd}
\end{equation}
\begin{equation}\label{me2}
\begin{tikzcd}
\, &\mathscr{G} v' \otimes v 
	\arrow[rightarrow]{dr}{\mathrm{d}\otimes 1}
	\arrow[rightarrow]{d}{1\otimes \gamma}
	\arrow[rightarrow,swap]{ddr}{\mathrm{d}\otimes\gamma}
&
\\
\mathscr{G}(v'\otimes v) 
		  \arrow[leftarrow]{r}{\mu}&
		 \mathscr{G} v' \otimes \mathscr{G}v 	
& \mathscr{G}^2 v' \otimes v 
	\arrow[rightarrow]{d}{1\otimes\gamma}
\\
& \mathscr{G}^2 v' \otimes \mathscr{G}^2v 
	\arrow[leftarrow]{u}{\mathrm{d}\otimes\mathrm{d}}
& \mathscr{G}^2 v' \otimes \mathscr{G}v 
	\arrow[rightarrow]{d}{\mu}
	\arrow[rightarrow]{l}{1\otimes \mathscr{G}\gamma}
\\
\mathscr{G}^2(v'\otimes v)
	\arrow[leftarrow]{uu}{\mathrm{d}}
& \mathscr{G}(\mathscr{G} v' \otimes \mathscr{G}v )
	\arrow[rightarrow]{l}{\mathscr{G}\mu}
	\arrow[leftarrow]{u}{\mu}
& \mathscr{G}(\mathscr{G} v' \otimes v )
	\arrow[rightarrow]{l}{\mathscr{G}(1\otimes\gamma)}
\end{tikzcd}
\end{equation}
\begin{example}
Comonads obtained by tensoring with a Hopf bimonoid in a braided monoidal category are all Hopf.
In particular, the three comonads in Example~\ref{threecomonads}
are Hopf.
\end{example}
\begin{example}\label{Hopfness_of_interior_op_ex}
The comonad $\CG$ on $\mathrm{Spn}(\mathcal{P}X)$ coming from an interior operator (as in Example~\ref{interioroperator}) is Hopf.
\end{example}
\begin{remark} We refer back to Example~\ref{comonoidalenrichedcats}. 
The notion of Hopf algebroid defined in Section 8 of \cite{60} is designed for modules
for a comonoidal $\CC$-category $\CA$; those modules are $\CC$-functors $M: \CA\to \CC$.
The monoidal structure on modules defined by the comonoidal structure on $\CA$ is
pointwise and the Hopfness implies that the closed structure is particularly simple.
However, the Hopf property in the present paper is designed for comodules
for a comonoidal $\CC$-category $\CA$. The two notions are different. 
\end{remark}

\section{Extension creation}\label{sec:extCreation}

In this section we show that Hopf comonadic pseudofunctors in $\mathrm{Caten}$ create left extensions.
This generalises results in \cite{Bruguieres2011, Chikhladze2010}. 

Recall (see \cite{Street1972a} for example) that a diagram
\begin{equation}\label{lKext}
\begin{aligned}
\xymatrix{
U \ar[rd]_{u}^(0.5){\phantom{a}}="1" \ar[rr]^{v}  && V \ar[ld]^{h}_(0.5){\phantom{a}}="2" \ar@{=>}"1";"2"^-{\kappa}
\\
& C 
}
\end{aligned}
\end{equation}
in a bicategory $\CW$ is said to exhibit $h$ as a {\em left (Kan) extension} of $u$ along $v$ when
for each 2-cell $\sigma : u \Rightarrow w\otimes v$ there exists a unique 2-cell 
$\tau : h \Rightarrow w$ whose pasted composite with $\kappa$ is $\sigma$.
Should it exist, we write $\mathrm{lan}(v,u)$ for such an $h$.

A pseudofunctor $\CH : \CW\to \CV$ between bicategories is said to {\em create left extensions}
when, given a span $C\xleftarrow{u} U\xrightarrow{v}V$ in $\CW$ for which a left extension of
$\CH u$ along $\CH v$ exists in $\CV$, there exists a left extension of $u$ along $v$ in $\CW$
which is taken by $\CH$ to a left extension of $\CH u$ along $\CH v$. 

\begin{theorem}\label{thm:createKan}
If the comonad $\mathscr{G}$ in $\mathrm{Caten}$ is left Hopf, then the forgetful pseudofunctor $\mathscr{U}:\mathscr{V}^\mathscr{G}\rightarrow \mathscr{V}$ creates left extensions.
\end{theorem}
\begin{proof}
Consider a span $G''\xleftarrow{(u,\gamma_u)}G \xrightarrow{(v,\gamma_v)} G'$ of coalgebras for which the underlying span $G''\xleftarrow{u}G \xrightarrow{v} G'$ has a left extension $k=\text{lan}(v, u)$ in $\CV$ exhibited by a 2-cell $\kappa : u \Rightarrow k\otimes v$.
The universal property of the left extension is that the function taking 2-cells $\bar \phi:k\Rightarrow l$
to 2-cells $\phi:u\Rightarrow l\otimes v$, defined by $\phi=(\bar{\phi} \otimes 1) \kappa$, is a bijection. 
In particular, there is a map $\gamma_k:k\rightarrow \mathscr{G}k$ corresponding to $u\xrightarrow{\gamma_u}\mathscr{G}u\xrightarrow{\mathscr{G}\kappa}\mathscr{G}(k\otimes v)\xrightarrow{\mathrm{v}_{k,v}^{-1}}\mathscr{G}k\otimes v$ such that the diagram below commutes.
\begin{equation}\label{eq:gammaKanprop}
\begin{tikzcd}
u \arrow[rightarrow]{r}{\gamma_u} &\mathscr{G}u 
  \arrow[rightarrow]{r}{\mathscr{G}\kappa} 
  &  \mathscr{G}(k\otimes v) 
		  \arrow[rightarrow]{d}{\mathrm{v}_{k,v}^{-1}}
	  \\
k\otimes v \arrow[leftarrow]{u}{\kappa}
	\arrow[rightarrow]{rr}{\gamma_k\otimes 1}
 &  & \mathscr{G}k\otimes v \\
\end{tikzcd}
\end{equation}
The obtained arrow $\gamma_k$ defines a coalgebra structure on $k$, where the compatibility with $\mathrm{d}$ and $\mathrm{e}$ follows from the following two commutative diagrams.
\begin{equation*}
\begin{tikzcd}
\, & u
	\arrow[rightarrow,swap]{ldd}{\kappa}
	\arrow[rightarrow]{dd}{\gamma}
	\arrow[rightarrow]{rd}{\gamma}
	\arrow[rightarrow]{rrrd}{\kappa}
&&&
\\
& 
& \mathscr{G} u 	
	\arrow[rightarrow]{ddr}{\mathscr{G}\kappa}
	\arrow[rightarrow,swap]{d}{\mathscr{G}\gamma}
&& k\otimes v
	\arrow[rightarrow]{ddd}{\gamma\otimes 1}
\\
k\otimes v
	\arrow[rightarrow,swap]{ddr}{\gamma\otimes 1}
& \mathscr{G} u 	
	\arrow[rightarrow]{d}{\mathscr{G}\kappa}
	\arrow[rightarrow]{r}{\mathrm{d}}
& \mathscr{G}^2 u
	\arrow[rightarrow,swap]{d}{\mathscr{G}^2\kappa}
&&
\\
& \mathscr{G}(k\otimes v)
	\arrow[rightarrow]{d}{\mathrm{v}^{-1}}
	\arrow[rightarrow]{r}{\mathrm{d}}
& \mathscr{G}^2(k\otimes v)
	\arrow[rightarrow,swap]{d}{\mathscr{G}\mathrm{v}^{-1}}
& \mathscr{G}(k\otimes v)
	\arrow[rightarrow]{dr}{\mathrm{v}^{-1}}
	\arrow[rightarrow]{dl}{\mathscr{G}(\gamma \otimes 1)}
&
\\
& \mathscr{G}k\otimes v
	\arrow[rightarrow,swap]{dr}{\mathrm{d}\otimes 1}
& \mathscr{G}(\mathscr{G}k\otimes v)
	\arrow[rightarrow]{d}{\mathrm{v}^{-1}}
&
& \mathscr{G}k\otimes v
	\arrow[rightarrow]{dll}{\mathscr{G}\gamma\otimes 1}
\\
&
& \mathscr{G}^2 k\otimes v
&&
\end{tikzcd}
\end{equation*}
\begin{equation*}
\begin{tikzcd}
u
	\arrow[rightarrow]{rrr}{\kappa}
	\arrow[rightarrow]{dr}{1}
	\arrow[rightarrow,swap]{rdd}{\gamma}
	\arrow[rightarrow,swap]{ddd}{\kappa}
&&
& k\otimes v
\\
& u
	\arrow[rightarrow]{urr}{\kappa}
& 
&
\\
& \mathscr{G} u 	
	\arrow[rightarrow,swap]{u}{\mathrm{e}}
	\arrow[rightarrow]{r}{\mathscr{G}\kappa}
& \mathscr{G}(k\otimes v)
	\arrow[rightarrow,swap]{uur}{\mathrm{e}}
	\arrow[rightarrow,swap]{dr}{\mathrm{v}^{-1}}
&
\\
k\otimes v
	\arrow[rightarrow,swap]{rrr}{\gamma\otimes 1}
&
& 
& \mathscr{G}k\otimes v
	\arrow[rightarrow,swap]{uuu}{\mathrm{e}\otimes 1}
\end{tikzcd}
\end{equation*}
The 2-cell $\kappa$ is a coalgebra morphism as seen by substituting $\mathrm{v}^{-1}$ in (\ref{eq:gammaKanprop}).

To see that $\kappa$ exhibits $(k,\gamma_k)$ as a left extension of $(u,\gamma_u)$ along $(v,\gamma_v)$, consider a coalgebra $(l,\gamma_l):G'\rightarrow G''$, and a coalgebra morphism $\phi:u\Rightarrow l\otimes v$. In $\mathscr{V}$, the left extension universal property gives $\bar\phi:k\rightarrow l$. Using the commuting diagram 
\begin{equation*}
\begin{tikzcd}
\, & u
	\arrow[rightarrow,swap]{ld}{\kappa}
	\arrow[rightarrow]{d}{\gamma}
	\arrow[rightarrow,swap]{rrd}{\phi}
	\arrow[rightarrow]{r}{\kappa}
& k\otimes v 	
	\arrow[rightarrow]{rd}{\bar{\phi}\otimes 1}
&& 
\\
k\otimes v
	\arrow[rightarrow,swap]{ddr}{\gamma\otimes 1}
& \mathscr{G} u 	
	\arrow[rightarrow]{d}{\mathscr{G}\kappa}
	\arrow[rightarrow]{rd}{\mathscr{G}\phi}
& 
& l\otimes v
	\arrow[rightarrow]{ld}{\gamma}
	\arrow[rightarrow]{ldd}{\gamma\otimes 1}
&
\\
& \mathscr{G}(k\otimes v)
	\arrow[rightarrow]{d}{\mathrm{v}^{-1}}
	\arrow[rightarrow,swap]{r}{\mathscr{G}(\bar{\phi}\otimes 1)}
& \mathscr{G}(l\otimes v)
	\arrow[rightarrow,swap]{d}{\mathrm{v}^{-1}}
& 
&
\\
& \mathscr{G}k\otimes v
	\arrow[rightarrow,swap]{r}{\mathscr{G}\bar{\phi}\otimes 1}
& \mathscr{G} l\otimes v
&
& 
\end{tikzcd}
\end{equation*}
we indeed see that $\bar\phi$ is a coalgebra morphism.
\end{proof}
\begin{remark}
Examination of the above proof allows us to obtain a more localised result.
To obtain that the left extension of $(u,\gamma_u)$ along $(v,\gamma_v)$ is
created we do not need the global invertibility of the fusion operators. 
We only need that the function
$$\CV(u,\mathrm{v}_{v',v}) : \CV(u,\CG v'\otimes v)\to \CV(u, \CG (v'\otimes v))$$
should be injective for all $v'\in \mathscr{V}(G'_0,G''_0)$ and surjective for $v'=k$, while the function 
$$\CV(u,\CG\mathrm{v}_{k,v}) : \CV(u,\CG(\CG k\otimes v))\to \CV(u, \CG\CG (k\otimes v))$$
should be injective.
\end{remark}
\begin{remark}
A dual of Theorem~\ref{thm:createKan} is that right Hopfness of $\mathscr{G}$ in $\mathrm{Caten}$ implies $\mathscr{U}:\mathscr{V}^\mathscr{G}\rightarrow \mathscr{V}$ creates left liftings.
\end{remark}

Recall that a morphism $m$ having a right adjoint in a bicategory is equivalent to the existence of a left extension $\text{lan}(m, 1_M)$ which is respected by $m$; that is, $m\circ\text{lan}(m, 1_M)\cong \text{lan}(m, m)$.
\begin{corollary}
With a Hopf-comonadic $\mathscr{U}:\mathscr{N}\rightarrow\mathscr{M}$, a morphism $n\in\mathscr{N}(N,N')$ has a right adjoint if and only if $\mathscr{U}n$ does.
\end{corollary}
\begin{proof}
Being a pseudofunctor, $\mathscr{U}$ preserves adjoints.

The other way around, assume $\CU n$ has a right adjoint, so both $\text{lan}(\CU n, 1_{\CU N})$ and $\text{lan}(\CU n, \CU n)$ exist. From the previous theorem, $\text{lan}(n, 1_N)$ exists and $n\circ\text{lan}(n, 1_N)$ is taken to $\CU (n\circ\text{lan}(n, 1_N))\cong \CU n\circ\text{lan}(\CU n, 1_{\CU N})\cong \text{lan}(\CU n, \CU n)$ which creates $\text{lan}(n, n)$.
\end{proof}

\section{Hopfness of comonadic base change}\label{Hbc}

We now show that Hopfness of the pseudofunctor inducing the base change passes to the base change itself.

\begin{theorem}\label{thm:HopfnessModU}
In the situation of Theorem~\ref{thm:comonadicModU}, let $\widetilde{\mathscr{R}}$ be the right adjoint of $\widetilde{\mathscr{U}}$.
If $\mathscr{R}$ preserves local colimits, and $\mathscr{G}$ is left Hopf, then the induced comonad $\widetilde{\mathscr{G}}:=\widetilde{\mathscr{U}}\circ\widetilde{\mathscr{R}}$ is right Hopf.
\end{theorem}
\begin{proof}
Again we consider the case when $\mathscr{X}$ is the terminal bicategory. 
Recall that we are looking at $\mathscr{N}=\mathscr{V}^\mathscr{G}\text{-Mod}$, 
and $\mathscr{M}=\mathscr{V}\text{-Mod}$.

At modules $M\in \mathscr{M}(\mathscr{U}\circ\mathscr{A},\mathscr{U}\circ\mathscr{B})$ and $N\in \mathscr{N}(\mathscr{B},\mathscr{C})$, the right fusion operator  $\mathrm{v}_{N,M}^\mathrm{r}$ for $\widetilde{\mathscr{G}}$ is 
given by the right column of the following diagram.
\begin{equation*}
\hspace{-0.5cm}
\begin{tikzpicture}[scale=1.5, every node/.style={scale=1},baseline=(current  bounding  box.center)]
\node (C1) at (-3.5,2.1) {
$\sum \mathscr{R} M_{B'}^A\otimes\mathscr{B}_B^{B'}\otimes N_{C}^B$};
\node (C2) at (0,2.1) {
$\sum \mathscr{R} M_{B}^A\otimes N_{C}^B$};
\node (C3) at (2.7,2.1) {
$(N \circ_{\mathscr{B}} \mathscr{R}M)_{C}^A$};
\node (A1) at (-3.5,0.7) {
$\sum\mathscr{R}M_{B'}^A\otimes\mathscr{B}_B^{B'}\otimes\mathscr{R}\mathscr{U}N_{C}^B$};
\node (A2) at (0,0.7) {
$\sum\mathscr{R}M_{B}^A\otimes\mathscr{R}\mathscr{U}N_{C}^B$};
\node (A3) at (2.7,0.7) {
$(\mathscr{R}\mathscr{U}N\circ_{\mathscr{B}}\mathscr{R}M)_{C}^A$};
\node (B1) at (-3.5,-0.7) {
$\sum \mathscr{R}(M_{B'}^A\otimes\mathscr{U}\mathscr{B}_B^{B'}\otimes\mathscr{U} N_{C}^B)$};
\node (B2) at (0,-0.7) {
$\sum \mathscr{R}(M_{B}^A\otimes \mathscr{U}N_{C}^B)$};
\node (B3) at (2.7,-0.7) {
$\mathscr{R}(\mathscr{U}N\circ_{\mathscr{U}\mathscr{B}}M)_{C}^A$};-
\path[transform canvas={yshift=1mm},->,font=\scriptsize,>=angle 90]
(C1) edge node[above] {$\widetilde{\mathscr{R}}\rho\otimes 1$} (C2);
\path[transform canvas={yshift=-1mm},->,font=\scriptsize,>=angle 90]
(C1)edge node[below] {$1\otimes \lambda$} (C2);
\path[->,font=\scriptsize,>=angle 90]
(C2) edge node[above] {$\text{coeq}$} (C3);
\path[transform canvas={yshift=1mm},->,font=\scriptsize,>=angle 90]
(A1) edge node[above] {$\widetilde{\mathscr{R}}\rho\otimes 1$} (A2);
\path[transform canvas={yshift=-1mm},->,font=\scriptsize,>=angle 90]
(A1)edge node[below] {$1\otimes \widetilde{\mathscr{R}}\mathscr{U}\lambda$} (A2);
\path[->,font=\scriptsize,>=angle 90]
(A2) edge node[above] {$\text{coeq}$} (A3);
\path[transform canvas={yshift=1mm},->,font=\scriptsize,>=angle 90]
(B1) edge node[above] {$\mathscr{R}(\rho\otimes 1)$} (B2);
\path[transform canvas={yshift=-1mm},->,font=\scriptsize,>=angle 90]
(B1) edge node[below] {$\mathscr{R}(1\otimes\mathscr{U}\lambda)$} (B2);
\path[->,font=\scriptsize,>=angle 90]
(B2) edge node[below] {$\mathscr{R}(\text{coeq})$}
			node[above] {$\text{coeq}$} (B3);
\path[transform canvas={xshift=15mm},->,font=\scriptsize,>=angle 90]
(C1) edge node[left] {$\sum 1\otimes 1\otimes \mathrm{g}$} (A1);
\path[transform canvas={xshift=-9mm},->,font=\scriptsize,>=angle 90]
(C2) edge node[right] {$\sum 1\otimes \mathrm{g}$} (A2);
\path[->,font=\scriptsize,>=angle 90]
(C3) edge node[left] {$(\mathrm{g}\circ 1)_{C}^A$} (A3);
\path[transform canvas={xshift=15mm},->,font=\scriptsize,>=angle 90]
(A1) edge node[left] {$\sum\mu^{(\mathscr{R})}(1\otimes\mathrm{g}\otimes 1)$} (B1);
\path[transform canvas={xshift=-9mm},->,font=\scriptsize,>=angle 90]
(A2) edge node[right] {$\sum \mu^{(\mathscr{R})}$} (B2);
\path[->,font=\scriptsize,>=angle 90]
(A3) edge node[left] {$\mu^{(\widetilde{\mathscr{R}})}$} (B3);
\end{tikzpicture}
\end{equation*}
The middle column is a coproduct of left fusion operators $\mathrm{v}_{M,N}^\mathrm{l}$ for $\mathscr{G}$ and so is invertible.
The left column can be rewritten as a sum of composites of the form  $\mathrm{v}_{\CB,N}^\mathrm{l} (\mathrm{v}_{M,\CB}^\mathrm{l}\otimes 1_{N^B_C})$ and so is invertible.
So $\mathrm{v}_{N,M}^\mathrm{r}$ is invertible.
\end{proof}
\begin{remark}
That ${\mathscr{G}}$ is left Hopf while $\widetilde{\mathscr{G}}$ is right Hopf in Theorem~\ref{thm:comonadicModU} is expected since our notational conventions lead to a morphism reversing inclusion $\CI : \CW^{\mathrm{op}} \to \CW \text{-}\mathrm{Mod}$; see Section 7.6 of \cite{Kelly2002}.  
\end{remark}
\begin{example} 
In the diagram \eqref{diag:adjs}
functors $U$ and $\Sigma$ forget differential and take a sum of all components of the graded abelian group. They are both part of Hopf-comonadic adjunctions, and by Theorem \ref{thm:createKan} create duals and cohoms. An abelian group $A$ has a dual if and only if it is finitely generated and projective (\cite{Street2007} has a proof). As a consequence of $\Sigma$ being Hopf-comonadic, a graded abelian group $A$ has a dual if and only if it has finitely many non-zero components each of which is finitely generated and projective. As a consequence of $U$ being Hopf-comonadic, a chain complex $A$ has a dual if and only if its underlying graded abelian group does. 
\end{example}
\begin{example}
Since the $U$ of diagram~\eqref{diag:adjs} is a left adjoint it preserves colimits, so by Theorem \ref{thm:comonadicModU} 
the change of base functor $\widetilde{\CU}$ creates Cauchy modules.
\end{example}

\begin{example}
As in Remark~\ref{CEgood}, suppose we have a cocomplete category $\CE$ with each slice category $\CE/E$ cartesian closed and we have a comonad $G$ on $\CE$ which preserves pullbacks. 
If furthermore $G$ preserves colimits then $\CR$
preserves local colimits as required by Theorem~\ref{thm:HopfnessModU} with $\CV = \mathrm{Spn}(\CE)$.

Let us look, in particular, at the case where $G = E\times - $, with the 
comonad structure defined by the diagonal comonoid structure on $E\in \CE$, so that $\CE^G = \CE/E$.
The comonad $\CG$ on $\mathrm{Spn}(\CE)$ in $\mathrm{Caten}$ has objects those of $\CE/E$; 
that is, pairs $(X,p)$ where $X\xra{p} E$ in $\CE$. The functor
\begin{eqnarray*}
\CG_{(X,p)}^{(X',p')} : \mathrm{Spn}(\CE)(X,X') \lra \mathrm{Spn}(\CE)(X,X')
\end{eqnarray*}
 takes a span $(u,S,v) : X\to X'$ to the span $(uk,\bar{S},vk) : X\to X'$ where $\bar{S}\xra{k} S$ is the equalizer of $pu$ and $p'v$. Let us use the method of generalised elements in $\CE$ to investigate the fusion operators: a $Z$-element of $A$ is a morphism $a : Z\to A$ and we write $a\in_{Z}A$. The $Z$-elements of $\bar{S}$ are the $x\in_ZS$
 such that $pux=p'vx$; so $(u,S,v) : (X,p)\to (X',p')$ in $\mathrm{Spn}(\CE/E)$ if and only if 
 $k : \bar{S}\hookrightarrow S$ is invertible.  We also see that 
 $$
\CG_{(X',p')}^{(X'',p'')}(u',S',v')\otimes \CG_{(X,p)}^{(X',p')}(u,S,v) = (u\mathrm{pr}_1,P,v'\mathrm{pr}_2)
$$
where a $Z$-element of $P$ is a pair $(x,x')\in_{Z}S\times S'$ such that 
\begin{eqnarray}\label{P}
vx=u'x' \ , & pux=p'vx\ , & p'u'x'=p''v'x' \ ,
\end{eqnarray}
while 
$$\CG_{(X,p)}^{(X'',p'')}(u',S',v')\otimes (u,S,v) = (u\mathrm{pr}_1,Q,v'\mathrm{pr}_2)$$   
where a $Z$-element of $Q$ is a pair $(x,x')\in_{Z}S\times S'$ such that 
\begin{eqnarray}\label{Q}
vx=u'x' \ , & pux=p''v'x' \ .
\end{eqnarray}
We see that $P$ is a subobject of $Q$ and the component of $\mu_{(X,p) (X'',p'')}^{(X'p')}$ is the inclusion $P \hookrightarrow Q$ as a span morphism. It is now an easy calculation to see that the fusion operators are invertible: for, when $(u,S,v)$
is a span over $E$, we have the equation $pu = p'v$ which, together with equations \eqref{Q}, implies equations \eqref{P}. 
So {\em $\CG$ is Hopf}. 
It follows from Theorems~\ref{thm:HopfnessModU}~and~\ref{thm:createKan} that {\em the
pseudofunctor 
$$\widetilde{\CU} : \mathrm{Spn}(\CE/E)\text{-}\mathrm{Mod} \lra \mathrm{Spn}(\CE)\text{-}\mathrm{Mod}$$
creates left extensions and left liftings}.
\end{example}

\begin{remark}\label{monoidcasecomonadicHopf}
Theorem~\ref{thm:HopfnessModU} restricts to the one-object case.
In the situation of Remark~\ref{monoidcasecomonadic}, if the monoidal comonad $\CG = \CR\circ \CU$ on $\CC$ is Hopf
and preserves reflexive coequalizers
then the comonad $\widetilde{\CG}= \widetilde{\CR}\circ \widetilde{\CU}$ in $\mathrm{Caten}$ is Hopf.
\end{remark}


\end{document}